\documentclass[11pt]{amsart}
\usepackage[numbers, square]{natbib}
\usepackage{amssymb}
\usepackage{amsmath}
\usepackage{amsfonts}
\usepackage{geometry}
\usepackage{amsthm}
\usepackage{hyperref}
\usepackage{wrapfig}

\usepackage[dvipsnames]{xcolor}

\providecommand{\U}[1]{\protect\rule{.1in}{.1in}}
\theoremstyle{plain}

\newtheorem{lemma}{Lemma}

\newtheorem{proposition}{Proposition}
\newtheorem{remark}{Remark}

\newtheorem{theorem}{Theorem}
\numberwithin{equation}{section}

 \DeclareMathOperator{\diam}{diam}

\DeclareMathOperator{\supp}{supp}

\newcommand{\pr}[1]{\left( #1 \right) }

\begin{document}
\title[ nodal sets ]
{Upper bounds of nodal sets for solutions of bi-Laplace equations: II}
\author{ Jiuyi Zhu}
\address{
Department of Mathematics\\
Louisiana State University\\
Baton Rouge, LA 70803, USA\\
Email:  zhu@math.lsu.edu }
\thanks{Zhu is supported in part by  NSF grant DMS-2154506 and DMS-2453348}
%\date{\today}
%\dedicatory{}

\begin{abstract}
 We investigate  the upper bounds of nodal sets for solutions of bi-Laplace equations without using frequency functions which play an essential role in the study of nodal sets in the celebrated work by Logunov \cite{Lo18}.
 We obtain some delicate monotonicity and propagation of smallness results by Carleman estimates.  A polynomial upper bound for the nodal sets of solutions is obtained.
\end{abstract}

\maketitle
\section{Introduction}

In this paper, we study the upper bounds of nodal sets for the solutions of bi-Laplace equations
\begin{equation}
\triangle^2_g u=W(x)u \quad \mbox{on} \ \mathcal{M},
\label{bi-Laplace-1}
\end{equation}
where $\mathcal{M}$ is a compact and smooth Riemannian manifold with dimensions $n\geq 2$.
Assume that $\|W\|_{L^\infty}\leq M$ for some large constant $M$. The nodal sets are the zero level sets of solutions. For the eigenfunctions of Laplace equations
\begin{equation}
\triangle_g \phi_\lambda+\lambda \phi_\lambda=0\quad \mbox{on} \ \mathcal{M},
\end{equation}
 a renowned conjecture asked by Yau \cite{Yau82} is that if the Hausdorff measure of nodal sets satisfies
$$ c\sqrt{\lambda}\leq H^{n-1}(x\in\mathcal{M}|\phi_\lambda=0)\leq C\sqrt{\lambda},   $$
where $c$, $C$ depend on the smooth manifold $\mathcal{M}$. The conjecture was answered in the seminal paper by Donnelly-Fefferman \cite{DF88} in real analytic manifolds.  Lin \cite{Lin91} provided a simpler proof for the upper bound for general second order elliptic equations on the analytic manifolds by studying frequency functions. We focus on the discussion of upper bounds of nodal sets 
on the smooth manifolds.  An exponential upper bound
$H^{n-1}(\{\phi_\lambda=0\})\leq C \lambda^{C\sqrt{\lambda}}$ was obtained by Hardt and Simon \cite{HS89}. Recently, an important improvement with a polynomial upper bound was shown by Logunov in \cite{Lo18} 
$$H^{n-1}(x\in\mathcal{M}| \phi_\lambda=0)\leq C \lambda^{\beta},$$
 where $\beta>\frac{1}{2}$ depends only on the dimension. A new combinatorial argument was developed in \cite{Lo18}, \cite{LM18} to investigate the doubling index of Laplace eigenfuncions. 
%An essential tool of studying doubling index is to use the monotonicity property of frequency functions. 
For the upper bounds of nodal sets  on smooth surfaces, Donnelly-Fefferman \cite{DF90} showed that  $H^{1}(\{x\in\mathcal{M}|\phi_\lambda=0\})\leq C \lambda^{\frac{3}{4}} $, See also a different proof given by Dong \cite{D92}.
%by using Carleman estimates and Calder\'on and Zygmund type decomposition. 
%A different proof based on frequency functions was given by Dong \cite{D92}. 
 Logunov-
Malinnikova \cite{LM18} were able to refine the upper bound to be $C \lambda^{\frac{3}{4}-\epsilon}$ for some small $\epsilon$ using the new combinatorial argument. 

%For the lower bounds of nodal sets of eigenfunctions,
% Logunov \cite{Lo18-1} answered Yau's conjecture and obtained the sharp lower bound of Laplace eigenfunctions on smooth manifolds. See also  e.g. \cite{B78}, \cite{S14}, \cite{SZ11}, \cite{CZ11}, \cite{M11}, \cite{LPS24} for related results on the lower bounds.

Some delicate monotonicity results of doubling index was shown using the almost monotonicity properties of frequency functions, which enables the combinatorial arguments in the breakthrough work in \cite{Lo18}.
 The main goal of this paper is to replace the tool of frequency functions by Carleman estimates in the study of upper bounds of nodal sets. For harmonic functions $\triangle u=0$, the frequency functions was defined as
 \begin{align}
     N(x, r)=\frac{r \int_{ B_r(x)} |\nabla u|^2}{\int_{\partial B_r(x)} u^2}.
     \label{fre-1}
 \end{align}
The frequency function (\ref{fre-1}), which is rescaling invariant for harmonic function, has deep geometric perspectives and describes the degree of leading polynomials if the function is locally expanded as the polynomial functions, See e.g.  \cite{GL86}, \cite{Lin91}, \cite{zhu16}. Due to its rescaling invariant and variational structure, the tool of frequency functions seems to be limited to some special types of partial differential equations. The desirable frequency functions for higher order elliptic equations or even second order elliptic equations with singular weight potentials seem not be available.
 Carleman estimates developed by \cite{C39} for  unique continuation property are weighted integral estimates. See e.g.  \cite{AKS62}, \cite{H83}, \cite{K07},  for the application of Carleman estimates for various partial differential equations. Compared with frequency functions, Carleman estimates are more flexible  and can be established for more general partial differential equations. Carleman estimates are also widely  used in inverse problems, control theorems, spectral theory, etc.

 %Logunov further studied the frequency function of harmonic functions and
%For the smooth manifolds, some progresses were made towards the upper bound of nodal sets.
% For higher dimensions $n\geq 3$,  Very recently,  

%For the lower bound, Logunov \cite{Lo1} answered Yau's conjecture and obtained the sharp lower bound for smooth manifolds. This breakthrough improved a polynomial lower bound obtained early by Colding and  Minicozzi \cite{CM}, Sogge and Zelditch \cite{SZ}. See also the same polynomial lower bound by different methods, e.g. \cite{HSo}, \cite{M}, \cite{S}.

%The upper bound of nodal sets was studied for general second order elliptic equations  in \cite{Lin}, \cite{HS}, \cite{HL1}, \cite{GR}, etc.  

Due to the lack of desirable rescaling invariant types of frequency functions for a wide range of elliptic equations, we will explore the tool of Carleman estimates in a refined way to study the upper bounds of nodal sets with the combinatorial arguments for doubling index. 
As a prototype, we aim to study the nodal sets for  (\ref{bi-Laplace-1}), which can be considered as the continuation of the early work  \cite{Zhu19} by the author.
%Let us briefly review the study of nodal sets for higher order elliptic equations. 
The Hausdorff dimension of nodal sets $\{x\in\mathcal{M}|u(x)=0\}$ of the solutions of higher order elliptic equations was shown be  not greater than $n-1$ in \cite{H20}.
The optimal upper bound of nodal sets for higher order elliptic equations was obtained by Lin and the author \cite{LZ22} in real analytic domains.
%see also  \cite{Ku95} for the study of some types of higher order elliptic equations. 
Complex analysis techniques were used in the  real analytic setting. An implicit finite upper bound of nodal sets for solutions in (\ref{bi-Laplace-1}) was obtained in \cite{Zhu19}. 
%A polynomial upper bound of  mixed nodal sets $\{x\in\mathcal{M}|u(x)=\triangle u(x)=0\}$ was shown in \cite{Zhu19}.
%studied by Han \cite{Han00}. It was shown in \cite{Han00} that thand the mixed nodal sets $\{u=\triangle u=0\}$ is not greater than $n-1$, and the Hausdorff dimension of the singular sets $\{D^\nu u=0 \ \mbox{for all} \ |\nu|<4\}$ is not greater than $n-2$. Especially, the Hausdorff measure of singular sets was studied by Han, Hardt and Lin in \cite{HHL03}. An implicit upper bound for the measure of singular sets in term of the doubling index was given.
 For the bi-Laplace equations on smooth manifolds,  we are able to show the following polynomial upper bounds for nodal sets of solutions in (\ref{bi-Laplace-1}).
\begin{theorem}
Let $u$ be the solutions of bi-Laplace equations (\ref{bi-Laplace-1}). There exists a positive constant $C$ that depends only on the manifold $\mathcal{M}$ such that
$$ H^{n-1}(x\in\mathcal{M}|u(x)=0)\leq CM^{\beta}, $$
where $\beta>\frac{1}{4}$ depends only on the dimension $n$.
\label{th1}
\end{theorem}

%The proof of Theorem \ref{th1} follows from the study of  the combinatorial arguments for doubling index in the celebrated paper by Logunov \cite{Lo18}. To carry out the  combinatorial arguments, the almost monotonicity of doubling index and propagation of smallness are necessary, which have been obtained by the frequency functions approach. Due to the lack of rescaling invariant type of frequency functions for bi-Laplace equation, we develop some new types of resecaling invariant Carleman estimates (i.e. Lemma  \ref{lemma-go})  to obtain these results.

The organization of the paper is as follows. 
In Section 2, we  develop a new type rescaling invariant Carleman estimate and obtain the almost monotonicity results for doubling index for solutions of bi-Laplace equations.
 Section 3 is devoted to the proof of the polynomial upper bound of nodal sets for solutions of bi-Laplace equations.
 The Appendix includes the proof of Caccioppoli type inequality with boundary terms for bi-Laplace equations and a simplex Lemma.The letters $C$, ${c}$,  ${c}_i$, $C_i$ denote positive constants that do not depend on $\lambda$ or $u$, and may vary from line to line. 

\section{monotonicity of doubling index}
%In all the aforementioned literature for the study of the upper bound of nodal sets of classical eigenfunctions, a crucial estimate is the following sharp quantitative doubling inequality established by Donnelly and Fefferman \cite{DF88},
%\begin{equation}
%\|\phi_\lambda\|_{\mathbb B_{2r}(x)}\leq e^{ C\sqrt{\lambda}}\|\phi_\lambda\|_{\mathbb B_{r}(x)}
%\label{doubleuse}
%\end{equation}
%for any $r>0$ and any $x\in \mathcal{M}$, where $\|\cdot\|_{\mathbb B_r(x_0)}$ denotes the $L^2$ norm on the ball $\mathbb B_r(x_0)$. Such optimal doubling inequalities provide the sharp upper bound for the frequency function and vanishing order for classical eigenfunctions. Roughly speaking, doubling inequalities retrieve global feature from local data. Those estimates are also widely used in inverse problems, control theorems, spectral theory, etc

%nodal sets for the solutions of bi-Laplace equations
%\begin{equation}
%\triangle^2 u=W(x)u \quad \mbox{in} \ \mathcal{M}
%\label{bi-Laplace-m}
%\end{equation}
%on a $n$-th dimensional smooth manifold without boundary.

For any $x_0\in \mathcal{M}$, let $r=d(x, \ x_0)=r(x)$ be the Riemannian distance from $x_0$ to $x$. $ B_r(x_0)$ is denoted as the
geodesic ball at $x_0$ with radius $r$. 
 We treat the Laplace-Beltrami operator on the manifold as an elliptic operator $\triangle $ in a domain in $\mathbb R^n$ most time. For the Euclidean distance $d(x,y)$ and Riemannian distance $d_g(x, y)$, there exists a small number $\epsilon>0$ such that
$$1-\epsilon\leq \frac{d_g(x, y)}{d(x,y)}\leq 1+\epsilon   $$
for $x, y\in  B_{r_0}$ with $r_0$ depending on $\epsilon$ and the manifold.
The symbol $\|\cdot\|$ denotes the $L^2$ norm. Specifically, $\|\cdot\|_{ B_r(x_0)}$ or $\|\cdot\|_r$ for short denotes the $L^2$ norm on the ball $ B_r(x_0)$.
%Our crucial tools to get the doubling inequality are the quantitative
% Carleman estimates.
 Carleman estimates are weighted integral inequalities with a weight function $e^{\tau\phi}$, where $\phi$ usually satisfies some convex condition. We construct the weight  function $\phi$ as follows.
 Set $$\phi=-g(\ln r(x)),$$ where $g(t)=t-e^{\epsilon t}$ for some small $0<\epsilon<1$ and $-\infty<t <T_0$. The positive constant $\epsilon$ is a fixed small number and $T_0$ is negative with $|T_0|$ large enough. One can check that
 \begin{equation}
 \lim_{t \to -\infty}-e^{-t} g''(t)=\infty \quad \mbox{and} \quad \lim_{t \to -\infty} g'(t)=1.
 \end{equation}
 %Such weight function $\phi$ was introduced by H\"ormander in \cite{H}.

We consider $P(x, \partial)=\frac{\partial}{\partial x_i}(a_{ij}(x) \partial_j )$, where $P(x, \partial)$ is a uniformly elliptic operator, $a_{ij}$ is Lipschitz continuous, and the summation is understood in the index $i, j$ by Einstein notation.
By a rescaling and argument, we may assume
that $a_{ij}(0)=\delta_{ij}$.  Thus, it holds that $|a_{ij}(x)-\delta_{ij}|\leq C|x|$.
The Carleman estimates (\ref{car-imp}) obtained in Lemma \ref{lemma-go}
are rescaling invariant, compared with the well-known Carleman estimates (\ref{carle-sec}), since there is no  extra  term $r^{\epsilon/2}$ in the norms in the right hand side of (\ref{car-imp}). The rescaling invariant structure is essential to obtain the almost monotonicity for doubling index in Lemma \ref{baa}.

\begin{lemma}
\label{lemma-go}
There exist positive constants $R_0$, $C$, which depend only on the $a_{ij}$ and $\epsilon$, such that, for any $f\in C^\infty_0(B_{R_0}\backslash \{0\})$ with $B_{R_0}\subset \mathbb R^n$, and any
large constant $\tau$ with $|\tau-k|\geq \frac{1}{3}$ for any $k\in \mathbb{N}$, then
\begin{align}
C\|e^{\tau \phi} r^2 \partial_i ( a_{ij} \partial_j  f)\| \geq 
\tau \|e^{\tau \phi} f\| +  \|e^{\tau \phi} r \nabla f\|.
\label{car-imp}
\end{align}
\end{lemma}
\begin{proof}
For $f\in C^\infty_0(B_{R_0}\backslash\{0\})$,  the following Carleman estimates hold, see e.g. \cite{H83}
 \begin{align}
 C\|  e^{\tau \phi} r^2\partial_i(a_{ij} \partial_j f)\|\geq \tau^\frac{3}{2}\|e^{\tau\phi} r^{\epsilon/2} f\|+\tau \|e^{\tau\phi} r^{1+\epsilon/2} \nabla f\|+\tau^{-\frac{1}{2}}\|e^{\tau\phi} r^{2+\epsilon/2} | \nabla^2 f\|
 \label{carle-sec}
 \end{align}
 for some $0<\epsilon<1$.
 We aim to obtain an estimate to compare the Carleman estimates for elliptic operators with constant coefficients and variable coefficients.
For any $R\leq c\tau^{-1}$ with $c$ small, thanks to (\ref{carle-sec}) and Lipschitz continuity of $a_{ij}$,
we have
\begin{align}
\|e^{\tau\phi}r^2(\triangle-P(x,\partial))f\|_{L^2(B_{R})} &\leq C \|e^{\tau\phi} r^3 \partial_{ij}f\|_{L^2(B_{R})}  + C\|e^{\tau\phi}r^2 |\nabla f| \|_{L^2(B_{R})}  \nonumber \\
&\leq c\tau^{-(1-{\epsilon}/{2})} \|e^{\tau\phi} r^{{\epsilon}/{2}} r^2 \partial_{ij}f\|_{L^2(B_{R})}  + c\|e^{\tau\phi}r^{(1+{\epsilon}/{2})}   |\nabla f| \|_{L^2(B_{R})}  \nonumber \\
&\leq c\tau^{-\frac{1}{2}} \|e^{\tau\phi} r^{{\epsilon}/{2}} r^2 \partial_{ij}f\|_{L^2(B_{R})}  + c\|e^{\tau\phi}r^{(1+{\epsilon}/{2})}   |\nabla f| \|_{L^2(B_{R})}  \nonumber \\ &\leq c \|e^{\tau\phi} r^2 P(x, \partial)f\|_{L^2(B_{R})},
\end{align}
where we used the fact that $0<\epsilon<1$ and $c$ is small constant.
Thus, we get that
\begin{align}
 \|e^{\tau\phi}r^2\triangle f\|_{L^2(B_{R})} \leq \frac{3}{2} \|e^{\tau\phi}r^2 P(x,\partial)f\|_{L^2(B_{R})}
 \label{compare-ell}
\end{align}
for $R\leq c\tau^{-1}$.
We also need the following Carleman estimate for the Laplace operator. For $f\in C^\infty_0(\mathbb R^n\backslash\{0\})$, it holds that
\begin{align}
\tau \|r^{-\tau} f\|+  \|r^{-\tau +1} \nabla f\| \leq C\|r^{-\tau +2} \triangle f\|,
\label{lap-car}
\end{align}
where  $|\tau-k|\geq \frac{1}{3}$ for any $k\in \mathbb{N}$. Its proof can be obtained as a decomposition for the Laplace operator into eigenspaces, see e.g. \cite{PW98}, \cite{MZ26}. Each $k$ corresponds to spherical harmonics of degree $k$.
By the definition of $\phi$, if $R_0$ is small, then  $r^{-\tau} \approx e^{\tau \phi}$ for $0<r<R_0$, where the notation $b\approx a$ denotes as $b=O(a)$.
It follows from  (\ref{lap-car}) that
\begin{align}
\tau \|e^{\tau\phi} f\|_{L^2(B_{R_0})}+  \|e^{\tau\phi} r\nabla f\|_{L^2(B_{R_0})} \leq C \|e^{\tau\phi } r^2\triangle f\|_{L^2(B_{R_0})}.
\label{lap-car-1}
\end{align}
Let $A_{R_1, R_2}=B_{R_2}\backslash B_{R_1}$.
Thanks to (\ref{carle-sec}), for any $f\in C^\infty_0( A_{c\tau^{-1}, R_0})$ and $0<\epsilon<1$, we obtain that
\begin{align}
 C\| e^{\tau \phi} r^2 \partial_i(a_{ij} \partial_j f)\|&\geq \tau^\frac{3}{2}\|e^{\tau\phi} r^{\epsilon/2} f\|  +\tau^{\frac{1}{2}} \|e^{\tau\phi} r^{1+\epsilon/2} \nabla f\| 
 \nonumber\\
 & \geq \tau \tau^{\frac{1-\epsilon}{2}} 
\|e^{\tau\phi} f\|_{L^2(A_{c\tau^{-1}, R_0})}+ \tau^{\frac{1-\epsilon}{2}}  \|e^{\tau\phi} r\nabla f\|_{L^2(A_{c\tau^{-1}, R_0})} 
\nonumber \\
 & \geq  \tau \|e^{\tau\phi} f\|_{L^2(A_{c\tau^{-1}, R_0})} + \|e^{\tau\phi} r\nabla f\|_{L^2(A_{c\tau^{-1}, R_0})} .
 \label{cut-ell}
\end{align}

Now we choose a smooth cut-off function such that $\eta(x)=1$ on $|x|\leq c$, but $\eta=0$ on $|x|\geq 2c$.
Let $\eta_\tau(x)=\eta(\tau x)$. For $f\in C^\infty_0(B_{R_0}\backslash\{0\})$,  it follows from  (\ref{compare-ell}), (\ref{lap-car-1}) and (\ref{cut-ell}) that 
\begin{align}
\tau \|e^{\tau \phi} f\| +  \|e^{\tau \phi} r \nabla f\| & \leq \tau \|e^{\tau \phi} \eta_\tau  f\|+ \tau\|e^{\tau \phi}(1- \eta_\tau) f\| + \|e^{\tau \phi} r\nabla(\eta_\tau f)\| +  \|e^{\tau \phi} r\nabla((1-\eta_\tau) f))\|
\nonumber \\
&\leq C \|e^{\tau \phi} r^2 \triangle (\eta_\tau f)\| + C\|e^{\tau \phi} r^2  \partial_j (a_{ij} \partial_i((1-\eta_\tau) f))\| \nonumber \\
&\leq C \|e^{\tau \phi} r^2 \partial_i ( a_{ij} \partial_j (\eta_\tau f))\| +C \|e^{\tau \phi} r^2  \partial_j (a_{ij} \partial_i((1-\eta_\tau) f))\|.
\end{align}
Furthermore, by the Carleman  estimates (\ref{carle-sec}), the last inequality  and the fact that $0<\epsilon<1$, we obtain
\begin{align}
\tau \|e^{\tau \phi} f\| +  \|e^{\tau \phi} r \nabla f\| 
&\leq C\|e^{\tau \phi} r^2 \partial_i ( a_{ij} \partial_j  f)\|+ C\sum_{|\alpha|\leq 1} \tau^{2-|\alpha|} \|e^{\tau \phi} r^2 D^\alpha f\|_{A_{c\tau^{-1},  2c\tau^{-1}}}\nonumber \\
&\leq C\|e^{\tau \phi} r^2 \partial_i ( a_{ij} \partial_j  f)\|+ C\sum_{|\alpha|\leq 1} \tau^{\epsilon/2} \|e^{\tau \phi} r^{\epsilon/2} r^{|\alpha|} D^\alpha f\|_{A_{c\tau^{-1},  2c\tau^{-1}}}\nonumber \\
&\leq C\|e^{\tau \phi} r^2 \partial_i ( a_{ij} \partial_j f)\|.
\end{align}
This completes the proof of (\ref{car-imp}).
\end{proof}

We iterate the Carleman estimates (\ref{car-imp}) for Laplace-Beltrami operator on a coordinate patch on the manifold $\mathcal{M}$. 
%Let $\tau-2$ in th right hand side of the Carleman estimates  (\ref{car-imp}). 
For any $x_0\in  \mathcal{M}$, $f\in C^\infty_0(B_{R_0}(x_0) \backslash\{x_0\})$ with $R_0$ small, using the fact that $r e^{\phi}\approx 1$,
we have
\begin{align}
 C\| r^4 e^{\tau \phi} \triangle^2 f\|&\geq 
 \tau\| e^{(\tau-2) \phi} \triangle f\| \geq c_1\tau\| e^{\tau \phi} r^2\triangle f\| \nonumber \\
 &\geq c_2\tau^2\|e^{\tau\phi} f\|.
\end{align}

If we choose $\tau>C(1+\|W\|_{L^\infty}^\frac{1}{2})$ and  $|\tau-k|\geq \frac{1}{3}$ for any $k\in \mathbb{N}$, by the triangle inequality,  we obtain that
 \begin{align}
 C\| r^4 e^{\tau \phi} (\triangle^2 f-W(x)f)\|\geq \tau^2\|e^{\tau\phi} f\|.
 \label{Carle2}
 \end{align}

We apply the Carleman estimates (\ref{Carle2}) to the solutions of bi-Laplace equations to obtain the almost monotonicity results.
Note that $\|W\|_{L^\infty}\leq M$ for $M>1$ sufficiently large.
We do a scaling for the bi-Laplace equations  (\ref{bi-Laplace-1}). Let
$\bar u(x)=u(\frac{x}{M^{1/4}}).$
Then $\bar u(x)$  satisfies
\begin{equation}
\triangle^2 \bar u=\bar W(x)\bar u,
\label{bi-Laplace}
\end{equation}
where $\bar W(x)=\frac{W(x)}{M}$. Thus, $ \|\bar W\|_{L^\infty}\leq 1$. For ease of notations, we still write $\bar u$ and $\bar W(x)$ in (\ref{bi-Laplace}) as $u$ and $W(x)$.
 We define the doubling index as
\begin{equation}
{N}( B_r(x))=\log_2 \frac{\|u\|_{L^\infty(2 B_r(x))}}{\|u\|_{L^\infty( B_r(x))}}.
\label{defined}
\end{equation}
 For a positive number $\rho$, $\rho B$ is denoted
as the ball scaled by a factor $\rho>0$ with the same center as $ B$. ${N}(x, r)={N}( B_r(x))$ is the double index for $u$ on the ball $ B(x, r)$. We will write ${N}(x, r)$ as $N(r)$ if the center of the ball is understood.

Thanks to the Carleman estimates (\ref{Carle2}), we are able to show the following almost monotonicity results for doubling index.
\begin{lemma}
For any small $\delta\in (0, \frac{1}{10})$, there exist $C(g)$  and $\bar R(\delta, g)$,  such that the solutions $u$ of (\ref{bi-Laplace}) satisfy
\begin{align}
 t^{(1-\delta)N(x, R)+C\log_2 \delta}\leq \frac{\|u\|_{L^\infty(B_{tR}(x))}} { \|u\|_{L^\infty(B_{R}(x))}} \leq t^{ (1+\delta)N(x, tR)+C\log_2 1/\delta}
 \label{mono-del}
\end{align}
 for any $R>0$, $B_{tR}(x)\subset B_{\bar R}(0)$ and $t>2$. Furthermore,  there exists  $N_0(\delta, g)=C\delta^{-1}\log_2 \delta^{-1}$ such that 
\begin{align}
 t^{(1-\delta)N(x, R)}\leq \frac{\|u\|_{L^\infty(B_{tR}(x))}} { \|u\|_{L^\infty(B_{R}(x))}} \leq t^{ (1+\delta)N(x, tR)}
 \label{mono-del-1}
\end{align}
for  $N(x, R)\geq N_0(\delta, g)$.
\label{baa}
\end{lemma}

\begin{proof}

Let us show the first inequality in (\ref{mono-del}). We observe that it is actually a three-ball inequality with a delicate dependence on the parameters.
Since the equation (\ref{bi-Laplace}) is translation invariant, we may consider $x$ as the origin.
We introduce a cut-off function $\eta(r)\in C^\infty_0( B_{2tR})$ with $tR<\frac{R_0}{3}$. Let $0<\eta(r)<1$ satisfy the following properties:
\begin{itemize}
\item $\eta(r)=0$ \ \ \mbox{if} \ $r(x)<(1-2\delta)R$ \ \mbox{or} \  $r(x)>(t-\delta)R$, \medskip
\item $\eta(r)=1$ \ \ \mbox{if} \ $(1-\delta)R<r(x)<(t-2\delta)R$, \medskip
\item $|\nabla^\alpha \eta|\leq \frac{C}{{(\delta R)}^{|\alpha|}}$
\end{itemize}
for $\alpha=(\alpha_1, \cdots, \alpha_n)$. Recall that $A_{R_1, R_2}=B_{R_2}\backslash B_{R_1}$. We also denote $ \|u\|_{L^2(A_{R_1, R_2})} $ as $\|u\|_{R_1, R_2}$.
Since the function $\eta u$ is supported in the annulus $A_{(1-2\delta)R, (t-\delta)R}$, applying the Carelman estimates (\ref{Carle2}) with $f=\eta u$, we have
\begin{align}
\tau^2  \| e^{\tau \phi} \eta u\|& \leq C\| r^4 e^{\tau \phi}(\triangle^2 (\eta u)-W(x)\eta u)\| \nonumber  \\
&=C \| r^4 e^{\tau \phi}[\triangle^2, \ \eta] u\|,
\end{align}
where we have used the equation (\ref{bi-Laplace}). Note that $[\triangle^2, \ \eta]$ is a three-order differential operator involving the derivatives of $\eta$.
The properties of $\eta$ yield that
\begin{align*}
\| e^{\tau \phi} u\|_{(1-\delta)R, (t-2\delta)R}&\leq C \left(\| e^{\tau \phi} u\|_{(1-2\delta)R, (1-\delta)R}+\| e^{\tau \phi} u\|_{(t-2\delta)R, (t-\delta)R} \right)\\
&+ C\left(\sum_{|\alpha|=1}^{3} \| r^{|\alpha|} e^{\tau \phi} \nabla^\alpha u \|_{(1-2\delta)R, (1-\delta)R}+\sum_{|\alpha|=1}^{3} \| r^{|\alpha|} e^{\tau \phi} \nabla^\alpha u \|_{(t-2\delta)R, (t-\delta)R}\right).
\end{align*}
By the fact that the weight function $\phi$ is radial and decreasing, we have
\begin{align}
 \| e^{\tau \phi} u\|_{(1-\delta)R, (t-2\delta)R}&\leq C \Big(e^{\tau \phi((1-2\delta)R)} \|  u\|_{(1-2\delta)R, (1-\delta)R}+e^{\tau \phi((t-2\delta)R)} \| u\|_{(t-2\delta)R, (t-\delta)R}\Big) \nonumber \\
&+ C\Big(e^{\tau \phi((1-2\delta)R)}(\sum_{|\alpha|=1}^{3} \| r^{|\alpha|}  \nabla^\alpha u \|_{(1-2\delta)R, (1-\delta)R} \nonumber \\ &+e^{\tau \phi((t-2\delta)R)} \sum_{|\alpha|=1}^{3} \| r^{|\alpha|}  \nabla^\alpha u \|_{(t-2\delta)R, (t-\delta)R}\Big).
\label{recall}
\end{align}
For the higher order elliptic equations
\begin{equation} (- \triangle)^m u+W(x) u=0,
\label{zhu}
\end{equation}
 the following interior Caccioppoli type inequality holds, see e.g. \cite{zhu18},
\begin{equation}
\sum^{2m-1}_{|\alpha|=0} \|r^{|\alpha|}\nabla^\alpha u\|_{c_3R,  c_2R} \leq C (\|W\|_{L^\infty}+1)^{2m-1} \|u\|_{c_4R,  c_1R}
\label{hihcac}
\end{equation}
 for all positive constants $0<c_4<c_3<c_2<c_1<1$.
Applying the estimate (\ref{hihcac}) with $m=2$ yields that
\begin{align}
\| r^{|\alpha|} \nabla^\alpha u\|_{(1-2\delta)R, (1-\delta)R}\leq CM^3 \| u\|_{R}
\label{caccio-2}
\end{align}
and
\begin{align}
\| r^{|\alpha|} \nabla^\alpha u\|_{(t-2\delta)R, (t-\delta)R}\leq  C M^3 \| u\|_{tR}
\label{caccio-3}
\end{align}
for all $1\leq |\alpha|\leq 3$. Assume $t-2\delta>2+\delta$. Therefore, it follows from (\ref{recall}), (\ref{caccio-2}) and (\ref{caccio-3})  that
\begin{align}
\|u\|_{L^2(A_{(1-2\delta)R, (2+\delta)R})}\leq C\Big( e^{\tau(\phi((1-2\delta)R)-\phi((2+\delta)R))} \|u\|_{L^2(B_{R})}+  e^{\tau(\phi((t-2\delta)R)-\phi((2+\delta)R))} \|u\|_{L^2(B_{tR})}\Big).
\label{inequal}
\end{align}

We denote the exponents as 
$$ \beta^1_R=\phi((1-2\delta)R)-\phi((2+\delta)R),\quad
 \beta^2_R=\phi((2+\delta)R)-\phi((t-2\delta)R).$$
The definition of $\phi$ indicates that 
 $\beta_1>0$ and $\beta_2>0$.
Adding $\|u\|_{L^2(B_{(1-2\delta)R})}$ to both sides of the inequality (\ref{inequal}) gives that
\begin{equation}
\|u\|_{L^2(B_{(2+\delta)R})}\leq C \big( e^{\tau\beta^1_R}\|u\|_{L^2(B_{R})}+  e^{-\tau\beta^2_R}\|u\|_{L^2(B_{tR})}   \big).
\end{equation}
We aim to  incorporate the second term in the right hand side of the last inequality into the left hand side. To this end, 
 we choose $\tau$ such that
$$C e^{-\tau\beta^2_R}\|u\|_{L^2(B_{tR})}\leq \frac{1}{2}\|u\|_{L^2(B_{(2+\delta)R})},   $$
which holds if
$$\tau\geq \frac{1}{\beta^2_R} \ln \frac{2C \|u\|_{L^2(B_{tR})}}{\|u\|_{L^2(B_{(2+\delta)R})} }.   $$
Then we obtain that
\begin{equation}
\|u\|_{L^2(B_{(2+\delta)R})}\leq Ce^{\tau\beta^1_R}\|u\|_{L^2(B_{R})}.
\label{substitute}
\end{equation}
Note that $\tau>C$ and $|\tau-k|\geq \frac{1}{3}$ for any $k\in \mathbb{N}$ is needed to apply the Carleman estimates (\ref{Carle2}) since $\|W\|_{L^\infty}\leq 1$ in (\ref{bi-Laplace}). We choose
$$ \tau=\lceil C+\frac{1}{\beta^2_R} \ln \frac{2C\|u\|_{L^2(B_{tR})}}{\|u\|_{L^2(B_{(2+\delta)R})} } \rceil+\frac{1}{3}, $$
where $\lceil f\rceil$ denotes as the integer part of $f$.
Substituting such $\tau$ in (\ref{substitute}) gives that
\begin{align}
\|u\|_{L^2(B_{(2+\delta)R})}^{\frac{\beta^2_R+\beta^1_R}{\beta^2_R}} \leq  e^{C\beta^1_R} C^{\frac{ \beta^1_R+\beta^2_R}{\beta^2_R}}\|u\|_{L^2(B_{tR})}^{\frac{\beta^1_R}{\beta^2_R}} \|u\|_{L^2(B_{R})}.
\end{align}
Raising exponent $\frac{\beta^2_R}{\beta^2_R+\beta^1_R}$ to both sides of the last inequality shows that
\begin{align}
\|u\|_{L^2(B_{(2+\delta)R})} \leq C
e^{\frac{C\beta^1_R \beta^2_R}{ \beta^1_R +\beta^2_R}} 
\|u\|_{L^2(B_{tR})}^{\frac{\beta^1_R}{\beta^1_R+\beta^2_R}} \|u\|_{L^2(B_{R})}^{\frac{\beta^2_R}{\beta^1_R+\beta^2_R}}.
\label{three-1}
\end{align}
%Set $\alpha_1={\frac{\beta^1_R}
%{\beta^1_R+\beta^2_R}}$. 
Notice that the last inequality is a three-ball inequality.
Next we estimate the value of ${\frac{\beta^1_R}{\beta^1_R+\beta^2_R}}$ as follows. From the definition of the weight function $\phi$, we have
\begin{align}
\beta^1_R+\beta^2_R &=\ln (t-2\delta)R-[(t-2\delta)R]^\epsilon- \ln (1-2\delta)R+[(1-2\delta)R]^\epsilon  \nonumber \\
&=\ln \frac{t-2\delta}{1-2\delta}-R^\epsilon[ (t-2\delta)^\varepsilon- (1-2\delta)^\epsilon ].
\end{align}
Since $\delta$ is small, we have  $\ln (t-2\delta) \sim\ln t-\frac{2\delta}{t}$, and $\ln (1-2\delta)\sim -{2\delta}$, where $a\sim b$ denotes as $a-b=o(b)$.
We also choose $(tR)^\epsilon\leq \frac{\delta}{100}$, then 
\begin{align}
    \beta^1_R+\beta^2_R \geq \ln t+c_1\delta.
\end{align}
Now we compute $\beta^1_R$. 
We have
\begin{align}
     \beta^1_R&=\ln (2+\delta)R- [ (2+\delta)R]^\epsilon-\ln (1-\delta)R+ [ (1-\delta)R]^\epsilon  \nonumber \\
     &=\ln \frac{2+\delta}{1-\delta}-R^\epsilon[ (2+\delta)^\epsilon- (1-\delta)^\epsilon ] \nonumber \\
     &\leq \ln2 +c_2 \delta,
\end{align}
where we used $\ln (2+\delta)\sim \ln 2+\frac{\delta}{2}$ and $(2R)^\epsilon\leq (\frac{2}{t})^\epsilon\frac{\delta}{100}$. We can choose $0<c_1<c_2$.
Then 
\begin{align}
   {\frac{\beta^1_R}{\beta^1_R+\beta^2_R}}&\leq \frac{ \ln2}{(\ln t  +c_1 \delta )\frac{1}{1+c_2\delta/\ln 2} }\nonumber \\
   &\leq \frac{\ln 2} { (\ln t+c_1\delta)(1-c_2\delta/\ln 2)} \nonumber \\
   &\leq \frac{\ln 2} { (1-c_2\delta/\ln 2)\ln t+c_1\delta/2}. \nonumber \\
\end{align}
Let $\delta_1\sim c_2\delta/\ln 2$. We obtain that
\begin{align}
   {\frac{\beta^1_R}{\beta^1_R+\beta^2_R}}\leq \frac{1}{(1-\delta_1)log_2 t}.
\end{align}
It follows from (\ref{three-1}) and elliptic estimates (i.e. $L^\infty$ norm can be controlled by a $L^2$ norm in a larger ball) that
\begin{align}
\|u\|_{L^\infty(B_{2R})} \leq C\delta_1^{-\frac{n}{2}}
\|u\|_{L^\infty(B_{tR})}^{\frac{1}{(1-\delta_1)log_2 t}} \|u\|_{L^\infty(B_{R})}^{\frac{(1-\delta_1)log_2 t-1}{(1-\delta_1)log_2 t}}.
\end{align}
Then we get 
\begin{align}
\pr{\frac{\|u\|_{L^\infty(B_{2R})}} { \|u\|_{L^\infty(B_{R})}}}^{(1-\delta_1)log_2 t}\leq  \delta_1^{-C log_2 t}\frac{\|u\|_{L^\infty(B_{tR})}} { \|u\|_{L^\infty(B_{R})}},
\end{align}
which is equivalent to 
\begin{align}
    t^{(1-\delta_1)\log_2 \frac{\|u\|_{L^\infty(B_{2R})}} { \|u\|_{L^\infty(B_{R})}} }\leq C t^{-C\log_2 \delta_1}\frac{\|u\|_{L^\infty(B_{tR})}} { \|u\|_{L^\infty(B_{R})}}.
\end{align}
Therefore, we arrive at
\begin{align}
    t^{(1-\delta_1)N(R)+C \log_2 \delta_1} \leq \frac{\|u\|_{L^\infty(B_{tR})}} { \|u\|_{L^\infty(B_{R})}}.
\end{align}
Identifying $\delta_1$ as $\delta$, we complete the proof of the left-hand side of (\ref{mono-del}). 

Next we prove the second inequality in the right hand side of (\ref{mono-del}) using the similar arguments. Note that this inequality is a three-ball inequality with some different parameters from the previous one. 
We introduce a cut-off function $\eta(r)\in C^\infty_0( B_{2tR})$ with $tR<\frac{R_0}{3}$. Let $0<\eta(r)<1$ satisfy the following properties:
\begin{itemize}
\item $\eta(r)=0$ \ \ \mbox{if} \ $r(x)<(1-2\delta)R$ \ \mbox{or} \  $r(x)>(2t-\delta)R$, \medskip
\item $\eta(r)=1$ \ \ \mbox{if} \ $(1-\delta)R<r(x)<(2t-2\delta)R$, \medskip
\item $|\nabla^\alpha \eta|\leq \frac{C}{{(\delta R)}^{|\alpha|}}$
\end{itemize}
for $\alpha=(\alpha_1, \cdots, \alpha_n)$. Note that the function $\eta u$ is supported in the annulus $A_{(1-2\delta)R, (2t-\delta)R}$. Substituting $f=\eta u$ in  Carelman estimates (\ref{Carle2}) to have 
\begin{align}
 \tau^2 \| e^{\tau \phi}\eta u\|& \leq C \| r^4 e^{\tau \phi}[\triangle^2, \ \eta] u\|,
\end{align}
where  the equation (\ref{bi-Laplace}) is used. 
By the properties of $\eta$, we have
\begin{align*}
\| e^{\tau \phi} u\|_{(1-\delta)R, (2t-2\delta)R}&\leq C \pr{\| e^{\tau \phi} u\|_{(1-2\delta)R, (1-\delta)R}+\| e^{\tau \phi} u\|_{(2t-2\delta)R, (2t-\delta)R} }\\
&+ C\pr{\sum_{|\alpha|=1}^{3} \| r^{|\alpha|} e^{\tau \phi} \nabla^\alpha u \|_{(1-2\delta)R, (1-\delta)R}+\sum_{|\alpha|=1}^{3} \| r^{|\alpha|} e^{\tau \phi} \nabla^\alpha u \|_{(2t-2\delta)R, (2t-\delta)R}}.
\end{align*}
The properties of  weight function $\phi$  yield that
\begin{align}
\| e^{\tau \phi} u\|_{(1-\delta)R, (2t-2\delta)R}&\leq C \pr{e^{\tau \phi((1-2\delta)R)} \|  u\|_{(1-2\delta)R, (1-\delta)R}+e^{\tau \phi((2t-2\delta)R)} \| u\|_{(2t-2\delta)R, (2t-\delta)R} }\nonumber \\
&+ C \Bigg(e^{\tau \phi((1-2\delta)R)}\sum_{|\alpha|=1}^{3} \| r^{|\alpha|}  \nabla^\alpha u \|_{(1-2\delta)R, (1-\delta)R}\nonumber \\ & +e^{\tau \phi((2t-2\delta)R)} \sum_{|\alpha|=1}^{3} \| r^{|\alpha|}  \nabla^\alpha u \|_{(2t-2\delta)R, (2t-\delta)R} \Bigg).
\label{recall-1}
\end{align}
 Assume $2t-2\delta>t+\delta$. Therefore, by Caccioppoli type inequality  (\ref{hihcac}), we get that
\begin{align}
\|u\|_{L^2(A_{(1-2\delta)R, (t+\delta)R})}\leq C \pr{ e^{\tau\pr{\phi((1-2\delta)R)-\phi((t+\delta)R)}} \|u\|_{L^2(B_{R})}+   e^{\tau(\phi((2t-2\delta)R)-\phi((t+\delta)R))} \|u\|_{L^2(B_{2tR})}}.
\label{inequal-1}
\end{align}
We choose parameters
$$ \beta^1_R=\phi((1-2\delta)R)-\phi((t+\delta)R),\quad 
\beta^2_R=\phi((t+\delta)R)-\phi((2t-2\delta)R).$$
The properties of weight function $\phi$ show that  $\beta^1_R>0$ and  $\beta^2_R>0$.
We add $\|u\|_{L^2(B_{(1-2\delta)R})}$ to both sides of the inequality (\ref{inequal-1}) to have
\begin{equation}
\|u\|_{L^2(B_{(t+\delta)R})}\leq C\big( e^{\tau\beta^1_R}\|u\|_{L^2(B_{R})}+ e^{-\tau\beta^2_R}\|u\|_{L^2(B_{2tR})}   \big).
\end{equation}
To incorporate the second term in the right hand side of the last inequality into the left hand side,
 we choose $\tau$ such that
$$Ce^{-\tau\beta^2_R}\|u\|_{L^2(B_{2tR})}\leq \frac{1}{2}\|u\|_{L^2(B_{(t+\delta)R})},   $$
which is true if
$$\tau\geq \frac{1}{\beta^2_R} \ln \frac{2C\|u\|_{L^2(B_{2tR})}}{\|u\|_{L^2(B_{(t+\delta)R})} }.   $$
Thus, it follows that
\begin{equation}
\|u\|_{L^2(B_{(t+\delta)R})}\leq C e^{\tau\beta^1_R}\|u\|_{L^2(B_{R})}.
\label{substitute-1}
\end{equation}
To apply the Carleman estimates (\ref{Carle2}), we choose
$$ \tau=\lceil C+\frac{1}{\beta^2_R} \ln \frac{2C\|u\|_{L^2(B_{2tR})}}{\|u\|_{L^2(B_{(t+\delta)R})} }\rceil+\frac{1}{3}. $$
Applying such $\tau$ in (\ref{substitute-1}) to have
\begin{align}
\|u\|_{L^2(B_{(t+\delta)R})}^{\frac{\beta^2_R+\beta^1_R}{\beta^2_R}} \leq e^{C\beta^1_R} C^{\frac{ \beta^1_R+\beta^2_R}{\beta^2_R}}\|u\|_{L^2(B_{2tR})}^{\frac{\beta^1_R}{\beta^2_R}} \|u\|_{L^2(B_{R})}.
\end{align}
We raise exponent $\frac{\beta^2_R}{\beta^2_R+\beta^1_R}$ to both sides of the last inequality to have
\begin{align}
\|u\|_{L^2(B_{(t+\delta)R})} \leq  Ce^{\frac{C\beta^1_R \beta^2_R}{\beta^1_R +\beta^2_R }}
\|u\|_{L^2(B_{2tR})}^{\frac{\beta^1_R}{\beta^1_R+\beta^2_R}} \|u\|_{L^2(B_{R})}^{\frac{\beta^2_R}{\beta^1_R+\beta^2_R}}.
\label{three-1-1-1}
\end{align}

We estimate the value of ${\frac{\beta^1_R}{\beta^1_R+\beta^2_R}}$ as follows. The definition of $\phi$ gives that
\begin{align}
\beta^1_R+\beta^2_R &=\ln (2t-2\delta)R-[(2t-2\delta)R]^\epsilon- \ln (1-2\delta)R +[(1-2\delta)R]^\epsilon  \nonumber \\
&=\ln \frac{2t-2\delta}{1-2\delta}-R^\epsilon[ (2t-2\delta)^\epsilon- (1-2\delta)^\epsilon ].
\end{align}
As $\delta$ is chosen to be small, we have  $\ln (2t-2\delta)\sim \ln 2t-\frac{\delta}{t}$. It also holds that $(2tR)^\varepsilon\leq \frac{\delta}{50}$. Then
\begin{align}
    \beta^1_R+\beta^2_R  
    %\ln 2t-\frac{\delta}{t} +2\delta -c\delta \nonumber \\ 
    \leq \ln 2t+c_1\delta.
\end{align}
Now we compute $\beta^2_R$ to
 have
\begin{align}
     \beta^2_R&=\ln (2t-2\delta)R- [ (2t-2\delta)R]^\epsilon-\ln (t+\delta)R+ [ (t+\delta)R]^\epsilon  \nonumber \\
     &=\ln \frac{2t-2\delta}{t+\delta}-R^\epsilon[ (2t-2\delta)^\epsilon- (t+\delta)^\epsilon ] \nonumber \\
     &\geq \ln2 -c_2 \delta,
\end{align}
where we used $\ln (2t-2\delta)\sim \ln 2t-\frac{\delta}{t}$ and $(2tR)^\epsilon\leq \frac{\delta}{50}$. We can also choose $0<c_1<c_2$ to have
\begin{align}
   {\frac{\beta^2_R}{\beta^1_R+\beta^2_R}}&\geq \frac{ \ln2}{(\ln 2t  +c_1 \delta )\frac{1}{1-c_2\delta/\ln 2} }\nonumber \\
   &\geq \frac{\ln 2} { (\ln 2t+c_1\delta)(1+c_2\delta/\ln 2)} \nonumber \\
   &\geq \frac{\ln 2} {(1+c_2\delta/\ln 2) \ln 2t+c_3\delta}
\end{align}
for some $c_3>c_2$.
Hence we obtain that
\begin{align}
{\frac{\beta^2_R}{\beta^1_R+\beta^2_R}}&\geq
\frac{1} { (1+c_2\delta/\ln 2)(1+\log_2 t)+c_3\delta/\ln 2}\nonumber \\
&\geq \frac{1} { (1+\delta_1)\log_2t+1}
\end{align}
by selecting $\delta_1\sim\frac{(2c_2+c_3)\delta }{\ln 2}$.
%Let $\delta_1=c_2\delta/\ln 2$. We have
%\begin{align}
%   {\frac{\beta^2_R}{\beta^1_R+\beta^2_R}}\geq \frac{1}{(1+\delta_1)log_2 t+1}.
%\end{align}
Then 
\begin{align}
     {\frac{\beta^1_R}{\beta^1_R+\beta^2_R}}\leq \frac{(1+\delta_1)log_2 t}{(1+\delta_1)log_2 t+1}.
\end{align}
It follows from (\ref{three-1-1-1}) and elliptic estimates that
\begin{align}
\|u\|_{L^\infty(B_{tR})} \leq  C\delta_1^{-\frac{n}{2}} \|u\|_{L^\infty(B_{2tR})}^{\frac{(1+\delta_1)log_2 t}{(1+\delta_1)log_2 t+1}} \|u\|_{L^\infty(B_{R})}^{\frac{1}{(1+\delta_1)log_2 t+1}},
\end{align}
which yields that
\begin{align}
\frac{\|u\|_{L^\infty(B_{tR})}} { \|u\|_{L^\infty(B_{R})}} \leq {\delta_1}^{-C \log_2 t}
(\frac{\|u\|_{L^\infty(B_{2tR})}} { \|u\|_{L^\infty(B_{tR})}})^{(1+\delta_1)log_2 t}.
\end{align}
We obtain that
\begin{align}
 \frac{\|u\|_{L^\infty(B_{tR})}} { \|u\|_{L^\infty(B_{R})}} \leq t^{-C \log_2 {\delta_1}}
(2^{N(tR)})^{(1+\delta_1)log_2 t}.
\end{align}
Therefore, we arrive at
 \begin{align}
 \frac{\|u\|_{L^\infty(B_{tR})}} { \|u\|_{L^\infty(B_{R})}} \leq t^{ (1+\delta_1)N(tR) -C\log_2 {\delta_1}},
\end{align}
which gives the right hand side of (\ref{mono-del}) by letting $\delta_1$ be $\delta$. 

Recall that we have chosen $tR<\frac{R_0}{3}$ and we choose $(tR)^\epsilon\leq \frac{\delta}{100}$. Thus, we can choose $tR\leq R_0 (\frac{\delta}{100})^{\frac{1}{\epsilon}}=\bar R (\delta, g)$ by fixing the value of $\epsilon$.
If we can choose $N_0\geq -C\delta^{-1}\log_2\delta$, it follows from (\ref{mono-del}) that
\begin{align*}
 t^{(1-2\delta)N(x, R)}\leq \frac{\|u\|_{L^\infty(B_{tR}(x))}} { \|u\|_{L^\infty(B_{R}(x))}}.
\end{align*}
By considering $2\delta$ as $\delta$ and $\bar R(\delta,g)$ as $\bar R(\delta/2,g)$, the first inequality in (\ref{mono-del-1}) is derived.

If $N(x,R)\geq N_0$, from (\ref{mono-del}), we learn that $(1-\delta)N_0+C\log_2 \delta\leq (1+\delta)N(x, tR)-C\log_2 \delta$.  Then $N(x, tR)\geq (1-2\delta)N_0+C\log_2  \delta.$ If we choose $N_0\geq- {C}{\delta}^{-1}\log_2 \delta $, thanks to  (\ref{mono-del}) again, we obtain that
\begin{align}
\frac{\|u\|_{L^\infty(B_{tR}(x))}} { \|u\|_{L^\infty(B_{R}(x))}} \leq t^{ (1+2\delta)N(x, tR)}.
\end{align}
By considering $2\delta$ as $\delta$ and $\bar R(\delta,g)$ as $\bar R(\delta/2,g)$, we get the second inequality in (\ref{mono-del-1}).

\begin{remark}
Applying the Carleman estimates (\ref{car-imp}) to second order elliptic equations $\partial_i ( a_{ij} \partial_j u)+b_i\partial_i u+V(x)u=0$, we can obtain the almost monotonicity results for the doubling index of solutions of the second order elliptic equations as Lemma \ref{baa}. 
\end{remark}

%Since $R\leq R_0$,  Thus, we further choose $R\leq R_0 (\frac{\delta}{100})^{\frac{2}{\varepsilon}}$. Then $t\leq (\frac{\delta}{100})^{-\frac{1}{\varepsilon}}R^{-1}_0$
\end{proof}

\section{Upper bounds of nodal sets}
%Let $n\geq 3$ in this section.
After those preparations, we follow the ideas of the combinatorial arguments in the celebrated work of \cite{Lo18} to prove the polynomial upper bounds of nodal sets for solutions of bi-Laplace equation in (\ref{bi-Laplace-1}).

Let $x_1, x_2, \cdots, x_{n+1}$ be the vertices of a simplex $S$ in $\mathbb R^n$. Denote $\diam(S)$ as the diameter of the simplex $S$. We use $width(S)$ to denote the minimum distance between two parallel hyperplanes that contain $S$. The symbol $w(S)$ is defined as the relative width of $S$:
$$ w(S)=\frac{width(S)}{diam(S)}. $$
We assume $w(S)>\gamma$ for some constant $\gamma$. In particular, $x_1, x_2, \cdots, x_{n+1}$ are assumed not to be on the same hyperplane. We denote $x_0$ as the barycenter of $S$, i.e. $x_0=\frac{1}{n+1}\sum^{n+1}_{i=1} x_i$. Roughly speaking, next lemma shows that the doubling index will accumulate at the barycenter of the simplex if the doubling index at the vertices $\{x_1, x_2, \cdots, x_{n+1}\}$ are large. Using the results of almost monotonicity of the double index in Lemma \ref{baa}, we can show the following simplex lemma.
\begin{lemma}
Let $ B_i:=B_{r_i}(x_i)$ be ball centered at $x_i$ with $r_i$ less than $\frac{K diam(S)}{2}$ for some $K=K(\gamma, n)$, $i=1, 2, \cdots, n+1$. There exist positive constants $c=c(\gamma, n)$, $C=C(\gamma, n)\geq K$, $r=r(\gamma, g)$ and $N_0=N_0(\gamma, g)$ such that if $S\subset B_r$ and $N( B_i)>N$ with $N>N_0$ for each $i$, $i=1, 2, \cdots, n+1$, then
\begin{align}
N(x_0, C\diam(S))>(1+c)N. 
\label{simplex-mon}
\end{align} 
\label{simplex}
\end{lemma}
Since  the arguments follow from \cite{Lo18}, we include the proof of the lemma in Appendix for the completeness of the presentation.
%\section{Simplex lemma} \label{sec:sl}
%Let $x_1, \dots , x_{n+1}$ be vertices of a simplex $S$ in $\mathbb{R}^n$. The symbol $diam(S)$ will denote the diameter of $S$ and
% by $width(S)$ we will denote the width of $S$, i.e. the minimum distance between a pair of parallel hyperplanes such that $S$ is contained between them. Define the relative width of $S$: $w(S)= width(S) / diam(S)  $.
% Let $a>0$ and  assume that $w(S) >a$. In particular we assume that $x_1, \dots , x_{n+1}$ do not lie on the same hyperplane. For the purposes of the paper there will be sufficient a particular choice of $a$, which depends on the dimension $n$ only, the choice will be specified in Section \ref{sec:nq}.
%Denote by $x_0$ the barycenter of $S$.
% Formulate lemma for harmonic functions first. After that  on manifold.
%\begin{lemma} \label{sl}
% Let $B_i$ be balls with centers at $x_i$ and radii not greater than $\frac{K}{2} diam(S)$, $i=1, \dots, n+1$, where $K=K(a,n)$ is from the Euclidean geometry lemma.
% There exist positive numbers $c=c(a,n)$, $C=C(a,n)\geq K$, $r=r(M,g,O,a)$, $N_0=N_0(M,g,O,a)$ such that
% if $S \subset B(O,r)$ and
%if $N(B_i)> N$ for each $x_i$, $i=1 \dots {n+1}$, where $N$ is a number greater than $N_0$,
% then $N(x_0, C diam(S))> N(1+c)$.

%\end{lemma} 

Next we consider Cauchy propagation of smallness for bi-Laplace equations in the half ball. The Cauchy propagation of smallness for Laplace operator was derived using the frequency function approach in \cite{ARRS09}. See also \cite{Lin91}. We are able to obtain these results using Carleman estimates.
Denote the half ball by ${B}^{+}_{R}=\{x| \ |x|<R, x_n>0\}$ and the ball on the boundary by ${\bf B}_{R}=\{x| \ |x|<R, x_n=0 \}$.
We consider the bi-Laplace equation
\begin{align}
    \triangle^2 u=W(x)u\quad \quad \mbox{in} \quad  B^+_3.
    \label{bi-ci}
\end{align} 
  We choose a weight function
$   \psi (x) = e^{sh(x)} $,
where $h(x) = -|x - b|,$ and $b = (0,...,0,-b_{n})$ with  some small positive $b_{n} > 0$. 

%(\textcolor{blue}{need to include $s$ which is $\frac{c}{\delta}$ if consider $B_\delta$})

\begin{lemma}\label{prop GCE}
 Let $s$ be a fixed large constant and $v\in H^2({B}^{+}_{3})$ 
 with $\supp v\subset {B}^{+}_{3}\cup {\bf B}_3$,  Then there exists a positive constant $C_{0}$  such that 
 %for any $v \in C^{\infty}(\mbb{B}^{+}_{1/2})$ and 
%$$
%\tau > C_{s}(1 + \|H\|^{\frac{2}{3}}_{L^{\infty}}),
%$$
%one has
 \begin{align}
     &\|e^{\tau\psi}( \triangle^2 v-W(x)v)\|_{L^{2}({B}^{+}_{3})}+ \tau^{\frac{3}{2}}s^2\|e^{\tau\psi}\triangle v\|_{L^{2} ({\bf B}_{3})} + 
 \tau^{\frac{1}{2}}s\|e^{\tau\psi}\nabla\triangle v\|_{L^{2} ({\bf B}_{3})}
   \nonumber \\
&+\tau^{3}s^4\|e^{\tau\psi}v\|_{L^{2}({\bf B}_{3})} + 
 \tau^{2}s^3\|e^{\tau\psi}\nabla v\|_{L^{2} ({\bf B}_{3})}   
    \geq C\tau^{3} s^4\|e^{\tau\psi} v\|_{L^{2}({B}^{+}_{3})}.
    \label{carleman-2}
    \end{align} 
\end{lemma}
\begin{proof}
The following global Carleman estimates was shown in e.g. \cite{Zhu23},

\begin{equation}\label{carleman}
\begin{aligned}
     &\|e^{\tau\psi} \partial_i(a_{ij} \partial_jv) \|_{L^{2}({B}^{+}_{3})}+ \tau^{\frac{3}{2}}s^{2}\|\psi^{\frac{3}{2}}e^{\tau\psi}v\|_{L^{2}({\bf B}_{3})} + 
 \tau^{\frac{1}{2}}s\|\psi^{\frac{1}{2}}e^{\tau\psi}\nabla v\|_{L^{2} ({\bf B}_{3})}
    \\
    &\geq C_{0}\tau^{\frac{3}{2}}s^{2}\|\psi^{\frac{3}{2}}e^{\tau\psi}v\|_{L^{2}({B}^{+}_{3})} + C_{0}\tau^{\frac{1}{2}}s\|\psi^{\frac{1}{2}}e^{\tau\psi}\nabla v\|_{L^{2}({B}^{+}_{3})},
    \end{aligned}
\end{equation}
where $\partial_i(a_{ij} \partial_jv)$ is uniformly elliptic  and $a_{ij}$ is Lipschitz continuous. Note that $s\approx \frac{C}{\delta}$ for some large constant $C$ if we replace the ball $B^+_3$ by $B^+_\delta$ in (\ref{carleman}). Thus, the Carleman estimate (\ref{carleman}) is rescaling invariant. 

%Since we do not need to consider the dependence of $s$ and $\psi$ is bounded above and below, we can write the Carleman estimates (\ref{carleman})  as
%\begin{equation}\label{carleman-1}
%\begin{aligned}
 %    &\|e^{\tau\psi} \partial_i(a_{ij} \partial_jv)\|_{L^{2}({B}^{+}_{3})}+ \tau^{\frac{3}{2}}\|e^{\tau\psi}v\|_{L^{2} ({\bf B}_{3})} + 
 %\tau^{\frac{1}{2}}\|e^{\tau\psi}\nabla v\|_{L^{2}({\bf B}_{3})}
%    \\
%    &\geq C_{s}\tau^{\frac{3}{2}}\|e^{\tau\psi}v\|_{L^{2}({B}^{+}_{3})} + C_{s}\tau^{\frac{1}{2}}\|e^{\tau\psi}\nabla v\|_{L^{2}({B}^{+}_{3})}.
%    \end{aligned}
%\end{equation}
Replacing $v$ by $\triangle_g v$ (we still write $\triangle_g$ as $\triangle$)
in the Carleman estimates (\ref{carleman}) to have
\begin{align}
     &\|e^{\tau\psi} \triangle^2 v \|_{L^{2}({B}^{+}_{3})}+ \tau^{\frac{3}{2}} s^2\|e^{\tau\psi}\triangle v\|_{L^{2} ({\bf B}_{3})} + 
 \tau^{\frac{1}{2}}s\|e^{\tau\psi}\nabla\triangle v\|_{L^{2} ({\bf B}_{3})}
   \nonumber \\
    &\geq C\tau^{\frac{3}{2}}s^2\|e^{\tau\psi}\triangle v\|_{L^{2}({B}^{+}_{3})}.
    \label{carleman-3}
    \end{align}
Applying the Carleman estimates (\ref{carleman}) to the right hand side of (\ref{carleman-3}), we get
   \begin{align}
     &\|e^{\tau\psi} \triangle^2 v \|_{L^{2}({B}^{+}_{3})}+ \tau^{\frac{3}{2}} s^2\|e^{\tau\psi}\triangle v\|_{L^{2} ({\bf B}_{3})} + 
 \tau^{\frac{1}{2}}s\|e^{\tau\psi}\nabla\triangle v\|_{L^{2} ({\bf B}_{3})}
   \nonumber \\
&+\tau^{3}s^4\|e^{\tau\psi}v\|_{L^{2}({\bf B}_{3})} + 
 \tau^{2}s^3\|e^{\tau\psi}\nabla v\|_{L^{2} ({\bf B}_{3})}   \nonumber \\
    &\geq C \tau^{3} s^4\|e^{\tau\psi} v\|_{L^{2}({B}^{+}_{3})}.
    \end{align} 
Taking $\tau>C(1+\|W\|^{\frac{1}{3}}_{L^\infty})$, by triangle inequality, we arrive at \eqref{carleman-2}.

\end{proof}

Note that $s\approx \frac{C}{\delta}$ for some large constant $C$ if we study the Carleman estimates (\ref{carleman-2}) in the half ball  $B^+_\delta$.  Since we consider the rescaled half ball $B^+_3$, we may skip the dependence of $s$. The Cauchy propagation of smallness for bi-Laplace equations is shown as follows.

\begin{lemma} Let $u$ be the solution of (\ref{bi-ci}). There exists $0<\kappa<1$ such that 
\begin{equation}\label{3ball-ish}
    \| u\|_{L^{2}({B}^{+}_{1})} \leq e^{C(1+\|W\|^{\frac{1}{3}}_{L^{\infty}})}\|u\|_{L^{2}({B}^{+}_{2})}^{\kappa} (\sum^3_{j=0}\|\partial^j_{x_n}u\|_{H^{3-j}(\bf{B}_{2})})^{1-\kappa}.
\end{equation}
\label{propa-1}
\end{lemma}

\begin{proof}
We choose a smooth radial cutoff function $\eta$ such that $\eta(x) = 1$ in ${B}^{+}_{3/2}$, $\eta(x) = 0$ outside of ${B}^{+}_{7/4}$, and $|\nabla^k \eta(x) | \leq C$ for some constant $C$ and $1\leq k\leq 4$.
 Applying the Carleman estimates (\ref{carleman-2}) 
with $v(x)=\eta u $ gives that
\begin{equation}\label{eta carl}
\begin{aligned}
   & \|e^{\tau\psi}( \triangle^2 (\eta u)-W(x)(\eta u))\|_{L^{2}({B}^{+}_{7/4})} + \tau^{\frac{3}{2}}\|e^{\tau\psi}\triangle (\eta u)\|_{L^{2} ({\bf B}_{7/4})} + 
 \tau^{\frac{1}{2}}\|e^{\tau\psi}\nabla\triangle (\eta u)\|_{L^{2} ({\bf B}_{7/4})}
   \nonumber \\
&+\tau^{3}\|e^{\tau\psi}(\eta u)\|_{L^{2}({\bf B}_{7/4})} + 
 \tau^{2}\|e^{\tau\psi}\nabla (\eta u)\|_{L^{2} ({\bf B}_{7/4})}   \nonumber \\
    &\geq C_{s}\tau^{3}\|e^{\tau\psi} (\eta u)\|_{L^{2}({B}^{+}_{7/4})}.
    \end{aligned}
\end{equation}
Since $\psi<1$, from \eqref{bi-ci}, we obtain that
\begin{equation*}
\begin{aligned}
    \|e^{\tau\psi}[\triangle^2, \eta]u\|_{L^{2}({B}^{+}_{7/4})} + 
     e^{\tau}\sum^3_{j=0}\|\partial^j_{x_n}u\|_{H^{3-j}(\bf{B}_{7/4})}
    \geq C\tau^3\|e^{\tau\psi} u\|_{L^{2}({B}^{+}_{3/2})}.
    \end{aligned}
\end{equation*}
Note that $[\triangle^2, \ \eta]$ is a three-order differential operator involving the derivatives of $\eta$. From the properties of $\eta$, we get
\begin{equation}\label{bounding-5}
\begin{aligned}
    \|e^{\tau\psi}u\|_{L^{2}({B}^{+}_{7/4} \setminus {B}^{+}_{3/2})}& + \|e^{\tau\psi}\nabla \triangle u\|_{L^{2}({B}^{+}_{7/4}\setminus {B}^{+}_{3/2})} + \|e^{\tau\psi}\nabla^2 u\|_{L^{2}({B}^{+}_{7/4}\setminus {B}^{+}_{3/2})} 
     \\&  + \|e^{\tau\psi}\nabla  u\|_{L^{2}({B}^{+}_{7/4}\setminus {B}^{+}_{3/2})} +e^{\tau}\sum^3_{j=0}\|\partial^j_{x_n}u\|_{H^{3-j}(\bf{B}_{7/4})}\\
   & \geq C\tau^3\|e^{\tau\psi} u\|_{L^{2}({B}^{+}_{3/2})}.
    \end{aligned}
\end{equation}
We aim to find the maximum and minimum of $\psi$ on the left hand side and the right hand side of \eqref{bounding-5}. By the definition of $h(x)$ and choosing $b_n$ to be appropriately small, we see that $\max_{{B}^{+}_{7/4} \setminus {B}^{+}_{3/2} }h(x) = -\sqrt{ (\frac{3}{2})^2+b^2_{n}}\leq -(\frac{4}{3}+b_n))$  and $\min_{{B}^{+}_{1}}h(x) = -(1 + b_{n})$.
Plugging  these bounds in (\ref{bounding-5}), we arrive at
\begin{equation}\label{bounding 6}
\begin{aligned}
    e^{\tau e^{-s(\frac 4 {3}+b_{n})}}\big(\|u\|_{L^{2}({B}^{+}_{ 7/4})} &+\|\nabla u\|_{L^{2}({B}^{+}_{ 7/4})} +  \|\nabla^2 u\|_{L^{2}({B}^{+}_{ 7/4})} + \|\nabla \triangle u\|_{L^{2}({B}^{+}_{ 7/4})}\big)
    \nonumber \\
&+e^{\tau}\sum^3_{j=0}\|\partial^j_{x_n}u\|_{H^{3-j}(\bf{B}_{ 7/4 })}
    \geq C\tau^3e^{\tau e^{-s(1 + b_{n})}}\| u\|_{L^{2}({B}^{+}_{1})}. 
    \end{aligned}
\end{equation}
Applying the Caccioppoli inequality with boundary terms (\ref{caccioppoli no bdry}) in Lemma \ref{corollary Cac} in the Appendix yields that
\begin{equation}\label{bounding 7}
\begin{aligned}
    e^{\tau e^{-s(\frac 4 {3}+b_{n})}}\|u\|_{L^{2}({B}^{+}_{2})} + e^{\tau}\sum^3_{j=0}\|\partial^j_{x_n}u\|_{H^{3-j}(\bf{B}_{2})}
    \geq Ce^{\tau e^{-s(1 + b_{n})}}\| u\|_{L^{2}({B}^{+}_{1})}. 
    \end{aligned}
\end{equation}
Let
$$
A_{1} = \|u\|_{L^{2}({B}^{+}_{2})},\quad
A_{2} = \sum^3_{j=0}\|\partial^j_{x_n}u\|_{H^{3-j}(\bf{B}_{2})},\quad 
A_{3} = \| u\|_{L^{2}({B}^{+}_{1})}.
$$
Set $p_{0} =  {  e^{-s(1+b_{n})}- e^{-s(\frac 4 {3}+b_{n})}}$ and $p_{1} = 1 - e^{-s(1+ b_{n})}$. Note that $p_0>0$ and $p_1>0$. We  get
\begin{equation}\label{B ineq}
    e^{-\tau p_{0}}A_{1} + e^{\tau p_{1}}A_{2} \geq CA_{3}.
\end{equation}
We aim to incorporate the $A_{1}$ term into the right hand side of (\ref{B ineq}). Thus, we choose
$$
e^{-\tau p_{0}}A_{1} \leq \frac{1}{2}CA_{3}.
$$
Hence the following assumption of $\tau$ is needed
\begin{align}\label{english}
\tau \geq -\frac{1}{p_0}\ln (\frac{CA_{3}}{2A_{1}}).
\end{align}
Thanks to \eqref{B ineq}, for such $\tau$, we have that 
\begin{equation}\label{B2 B3 ineq}
    e^{\tau p_1}A_{2} \geq \frac{1}{2}CA_{3}.
\end{equation}
Note that $\tau > C(1 +\|W\|^{\frac{1}{3}}_{L^\infty})$ is chosen in the Carleman estimates (\ref{carleman-2}) and the assumption  in (\ref{english}). We select
\begin{equation*}
    \tau = C(1+\|W\|^{\frac{1}{3}}_{L^{\infty}} ) - \frac{1}{p_0}\ln (\frac{CA_{3}}{2A_{1}}).
\end{equation*}
Substituting this choice of $\tau$ into \eqref{B2 B3 ineq} leads that
\begin{equation*}
    e^{C(1+\|W\|^{\frac{1}{3}}_{L^{\infty}})}A_{1}^{\frac{p_1}{p_1 + p_0}} A_{2}^{\frac{p_0}{p_1 + p_0}} \geq CA_{3}.
\end{equation*}
Let $\kappa = \frac{p_1}{p_1 +p_0}$.  We have $0<\kappa<1$. By the definition of $A_1, A_2$ and  $A_3$,  the last inequality implies that
\begin{equation}
    \| u\|_{L^{2}({B}^{+}_{1})} \leq e^{C(1+\|W\|^{\frac{1}{3}}_{L^{\infty}})}\|u\|_{L^{2}({B}^{+}_{2})}^{\kappa} (\sum^3_{j=0}\|\partial^j_{x_n}u\|_{H^{3-j}(\bf{B}_{2})})^{1-\kappa}.
\end{equation}
  Therefore, the lemma is completed.
\end{proof}

Given a cube $Q$, we introduce the maximal doubling index $N(Q)$,
$$ N(Q)=\sup_{x\in Q, \ r\in(0, diam(Q))} N(x, r). $$
The maximal doubling index $N(Q)$ of the cube  is more convenient in applications. If a cube $q\subset Q$, it holds that $N(q)\leq N(Q)$.
If a cube $q\subset \cup_i Q_i$ with $\diam(Q_i)\geq \diam q$, it holds that $N(Q_i)\geq N(q)$ for some $Q_i$.

\begin{lemma}
Let $Q$ be a cube $[-R, \ R]^n$ in $\mathbb R^n$. Divide $Q$ into $(2A+1)^n$ equal subcubes $q_i$ with side length $\frac{2R}{2A+1}$. Let $\{q_{i, 0}\}$ be the subcubes with nonempty intersection with the hyperplane $\{x_n=0\}$ and $N(Q)\leq  N$. If $A>A_0$ and $N>N_0$ for some $A_0$ and $N_0$, then there exists at least a subcube $q_{i, 0}$ such that $N(q_{i, 0})\leq \frac{N}{2} $.
\label{lemdou}
\end{lemma}
\begin{proof}We prove it by contradiction. 
For each $q_{i, 0}$, assume that there exist some point $x_i \in q_{i, 0}$ and $r_i<\diam(q_{i, 0})$ such that $N(x_i, r_i)\geq N/2$ for $N\geq N_0$.
By scaling, we may assume that $R=\frac{1}{2}$. 
Let $\|u\|_{L^\infty(B_{1/4})}=M_0$.  We have $$\|u\|_{L^\infty(B_{1/8}(x_i))}\leq M_0$$ if $x_i\in B_{1/8}$ since
$  B_{1/8}(x_i)\subset B_{1/4}$. 
From the assumption $N(x_i, r_i)\geq N/2$ and the almost monotonicity of the doubling index in (\ref{mono-del}).  We have
\begin{align}
  (\frac{2A+1}{64\sqrt{n}})^{N/3}\leq  (\frac{2A+1}{64\sqrt{n}})^{N(x_i, \frac{8\sqrt{n}}{2A+1})}\leq \frac{\|u\|_{L^\infty(B_{1/8}(x_i) )}}{ \|u\|_{L^\infty(B_{\frac{8\sqrt{n}}{2A+1}}(x_i) )}}.
\end{align}
As $x_i\in q_{i,0}$, we further get
\begin{align}
\|u\|_{L^\infty(4 q_{i, 0} )}&\leq \|u\|_{L^\infty(B_{\frac{8\sqrt{n}}{2A+1}}(x_i) )}\leq  \|u\|_{L^\infty(B_{1/8}(x_i) )}(\frac{64\sqrt{n}}{2A+1})^{\frac{N}{3}} \nonumber \\
&\leq 2^{-CN\log A} M_0,
\label{useback}
\end{align}
where  $A\geq A_0$ for some large $A_0$. By elliptic estimates, we have
\begin{align}
  \|u\|_{W^{3,\infty }(2q_{i,0})}  \leq  C(2A+1)^{\frac{6+n}{2} } 
  \|u\|_{L^2(4q_{i,0})}  \leq   C(2A+1)^3    \|u\|_{L^\infty(4q_{i,0})}.
\end{align}

Let $\tilde{\Gamma}=B_{1/8}\cap \{x_n=0\}$. 
%Summing up all the cubes $q_{i,0}$ with intersection with $\tilde{\Gamma}$, 
The last inequality yields that
\begin{align}
 \sum^3_{j=0}\|\partial^j_{x_n}u\|_{W^{3-j, \infty}(\tilde{\Gamma})}  \leq C (2A+1)^{3}\sup_{q_{i,0}\subset \tilde{\Gamma}}\|u\|_{L^\infty(4 q_{i,0})}  
\leq e^{-CN\log A}M_0,
\end{align}
where we used (\ref{useback}) in the second inequality. Note that $\|u\|_{L^2(B^+_{1/4})}\leq CM_0$. By scaling and using the result of the propagation of smallness in (\ref{3ball-ish}) in Lemma \ref{propa-1}, we have
\begin{align}
\|u\|_{L^2( B^+_{1/16})}\leq  e^{-CN\log A} M_0.
\end{align}
We select a ball $ B_{1/64}(p)\subset B^+_{1/16}$.  Hence, by elliptic estimates,
\begin{align}
\|u\|_{L^\infty(  B_{1/64}(p))}\leq e^{-CN\log A} M_0.
\end{align}
Since $B_{1/4}\subset B_{5/16}(p)$, it holds that $\|u\|_{L^\infty ( B_{5/16}(p))}\geq M_0$. Thus, we show that
\begin{align}
\frac{ \|u\|_{L^\infty ( B_{5/16}(p))}} {\|u\|_{L^\infty(  B_{1/64}(p))}}\geq e^{CN \log A}.
\label{one-1}
\end{align}
The almost monotonicity result (\ref{mono-del}) implies that
\begin{align}
\frac{ \|u\|_{L^\infty (B_{5/16}(p))}} {\|u\|_{L^\infty( B_{1/64}(p))}  }\leq  (20)^{2\tilde{N}},
\label{two-1}
\end{align}
where $\tilde{N}$ is the doubling index in $ B_{{5}/{16}}(p)$. Therefore, it follows from (\ref{one-1}) and (\ref{two-1}) that
\begin{align}
\tilde{N}\geq 2N\geq 2 N(Q)
\end{align}
if $A$ is chosen to be large enough. Hence we arrive at a contradiction, which finishes the proof of the Lemma.
\end{proof}

Following the arguments in \cite{Lo18} and the almost monotonicity results for the doubling index in Lemma \ref{baa}, hyperplane Lemma \ref{lemdou}, and simplex Lemma \ref{simplex}, the following conclusion holds.
\begin{lemma}
If $Q$ is partitioned into equal sub-cubes $A^n$, where $A$ depends on $n$, then the number of sub-cubes with doubling index greater than
$\max\{ \frac{N(Q)}{1+c}, \  N_0\}$ is less than $\frac{1}{2}A^{n-1}$ for some $c$ dependent on $n$ and some fixed constant $N_0$.
\label{lem12}
\end{lemma}

Thanks to Lemma \ref{lem12} and the implicit finite upper bound results of nodal sets for bi-Laplace in \cite{Zhu19}, 
 we give the bounds of the nodal set $\{x\in Q| u=0\}$ for solutions of bi-Laplace equation (\ref{bi-Laplace}) in a small cube by iterating arguments.
 
\begin{proposition}
Let $N(Q)$ be the maximal doubling index of the cube $Q$ for the solutions $u$ in (\ref{bi-Laplace}). There exist positive constant $r$, $C$ and ${\alpha_0}$ such that for any solutions $u$ on $\mathcal{M}$ and $Q\subset  B_r$,
\begin{equation}
H^{n-1}(\{u=0\}\cap Q)\leq C diam^{n-1}(Q)N^{{\alpha_0}}(Q),
\label{prove}
\end{equation}
where ${{\alpha}_0}$ depends only on $n$.
\label{pro2}
\end{proposition}

\begin{proof}
 For any solutions $u$ in (\ref{bi-Laplace}),  we consider those solutions such that $N(Q)\leq N$.
Define the function
\begin{equation} F(N)=\sup_{N_{u}(Q)\leq N}\frac{ H^{n-1}(\{u=0\}\cap Q)}{ \diam^{n-1}(Q)}.
\label{defnew} \end{equation}
 By Theorem 4  in \cite{Zhu19}, it holds that $F(N)<\infty$ for any given $N$.
We aim to show that
$$  F(N)\leq C N^{{\alpha_0}} $$
for some ${\alpha_0}$ depending only on $n$,  which provides the proof of the Proposition. 
%As shown in \cite{Han00} for higher order elliptic equations, the Hausdorff dimension of the  sets $\{ u=0 \}$ is not greater than $n-1$. 
%The Hausdorff dimension of nodal sets  $\{ u=0 \}$ is no more than $n-1$. Such stratification can also be observed in Lemma \ref{haudim} in Section 6.
% We assume that $u$ and $v$ has the same co-dimension one zero sets in $Q$. Otherwise, $H^{n-1}(\{u=0\}\cap Q)=0$, then the proposition follows immediately. If there exist $x_0$ such that $u(x_0)=0$ in $Q$, then $N_{u}(Q)\geq 1$. 
% Thanks to Theorem 4 in \cite{Zhu19}, we have shown  that $F(N)< \infty$ for $N\leq N_0$. 
 %We first claim that if
%\begin{equation} F(N)> 3AF(\frac{N}{1+c}),
%\label{claim}
%\end{equation}
%then there holds $N\leq N_0$, where the constant $A$, $c$ are those in the last lemma and $N_0$ depends on the manifold $\mathcal{M}$. If  $F(N)$ is almost attained in (\ref{defnew}), then
%\begin{equation}
%\frac{ H^{n-1}(\{u=0\}\cap Q)}{ \diam^{n-1}(Q)}>\frac{5}{6} F(N),
%\label{contra}
%\end{equation}
%where  $N_{u}(Q)\leq N$. 
We divide $Q$ into $A^n$ equal subcubes $q_i$, $i=1, 2, \cdots, A^n$, then split $q_i$ into two groups
$$G_1= \{ q_i|\frac{N}{1+c}\leq N(q_i)\leq N\}   $$
and
$$ G_2=\{ q_i|N(q_i)<\frac{N}{1+c} \}.   $$
  Hence it holds that
\begin{align}
H^{n-1}(\{u=0\}\cap Q)&\leq \sum_{q_i\in G_1} H^{n-1}(\{u=0\}\cap q_i)+\sum_{q_i\in G_2} H^{n-1}(\{u=0\}\cap q_i) \nonumber \\
&\leq |G_1|F(N) \frac{ \diam^{n-1}(Q)}{A^{n-1}}+|G_2| F(\frac{N}{1+c})  \frac{ \diam^{n-1}(Q)}{A^{n-1}} \nonumber \\
&=I_1+I_2,
\end{align}
where $|G_i|$ denotes the number of subcubes in $G_i$. With aid of  Lemma \ref{lem12}, the number of subcubes in  $G_1$ is less than $\frac{1}{2}A^{n-1}$ if $N>N_0$. Then we have
\begin{equation}
I_1\leq \frac{1}{2} F(N) \diam^{n-1}(Q).
\label{II1}
\end{equation}
%If (\ref{claim}) holds, it follows that
%\begin{equation}
%I_2\leq |G_2|\frac{F(N)}{3A} \frac{ \diam^{n-1}(Q)}{A^{n-1}}.
%\end{equation}
Thus, we obtain
\begin{equation}
F(N)\leq 2 |G_2|   F(\frac{N}{1+c})/{A^{n-1}}.
\end{equation}
Since $|G_2|\leq A^n$, then it holds that
%The combination of the estimates  (\ref{II1}) and (\ref{II2}) gives that
%\begin{equation}
%I_1+I_2\leq \frac{5}{6} F(N)\diam^{n-1}(Q),
%\end{equation}
%which contradicts with  (\ref{contra}). Therefore,  the claim is verified. 
%That is, if  $N\geq N_0$, we have
\begin{equation} F(N)\leq 2AF(\frac{N}{1+c})
\label{result}
\end{equation}
for $N\geq N_0$.
Let $\frac{N}{(1+c)^k}=N_0$. We iterate the estimate (\ref{result}) $k$ times to get
\begin{align*}
F(N)&\leq (2A)^k F(\frac{N}{(1+c)^k}) \\
&=(1+c)^{(\log_{1+c} {2A}) (\log_{1+c} \frac{N}{N_0})} F(N_0)\\
&=(\frac{N}{N_0})^{(\log_{1+c} 2A)}F(N_0).
\end{align*}
Thus, the estimate (\ref{prove}) is obtained for $N\geq N_0$. If $N\leq N_0$, it follows from Theorem 4  in \cite{Zhu19} that
$F(N)\leq C(N_0)
$
for some $C$ that depends on $N_0$. Therefore,  the proposition is completed.

\end{proof}

Proposition \ref{pro2} leads to the polynomial upper bounds of nodal sets for solutions of  bi-Laplace equations in (\ref{bi-Laplace}).
\begin{proof}[ Proof of Theorem \ref{th1}]
We have shown  the following  doubling inequality in Theorem 3 in \cite{Zhu19},
\begin{equation}
\|u\|_{L^\infty( B_{2r}(x))}\leq e^{ CM^{\frac{1}{3}}}\|u\|_{L^\infty( B_{r}(x))}
\label{haha-my}
\end{equation}
for any $x\in \mathcal{M}$ and any $0<r<r_0$, where $r_0$ depends only on the manifold $\mathcal{M}$.
From the definition of doubling index in (\ref{defined}), it follows from (\ref{haha-my}) that
$$N(x, r)\leq CM^{\frac{1}{3}}$$
for $M$  large and for any $x\in \mathcal{M}$ and $0<r<r_0$. Thus, the doubling index $N(Q)\leq CM^{\frac{1}{3}}$ in any cube $Q$. For bi-Laplace equations (\ref{bi-Laplace}) in the cube $Q\subset B_r$ with $0<r\leq \frac{r_0}{{M}^{1/4}}$, the last proposition implies that
$$ H^{n-1}(\{u=0\}\cap Q)\leq C  M^{\frac{{\alpha_0}}{3}-\frac{n-1}{4}}.  $$
Since the manifold $\mathcal{M}$ is compact, we can cover the manifold with $CM^\frac{n}{4}$ number of balls with radius $r=\frac{r_0}{{M}^{1/4}}$. Therefore, we arrive at
$$ H^{n-1}(x\in \mathcal{M}|\{u=0\})\leq C  M^{\frac{{\alpha_0}}{3}+\frac{1}{4}}.  $$
This gives the proof of Theorem \ref{th1} by letting $\beta={\frac{{\alpha_0}}{3}+\frac{1}{4}}$.
\end{proof}

\section{Appendix}

In this appendix, we prove a Cacciopolli type inequality with boundary terms for bi-Laplace equations and the simplex Lemma \ref{simplex}. The following Cacciopolli inequality with boundary terms seems to be unknown in the literature for the bi-Laplace equations. The Cacciopolli inequality  (\ref{caccioppoli no bdry}) is used in 
Lemma \ref{propa-1}.

%By fixing $\delta$ small enough and $t$ large, the proof is arrived.
% Third, we choose 
%$R=R(\varepsilon,M,g,O) >0$ and $N_0=N_0(\varepsilon,M,g,O)$ %such that Lemma \ref{ln3} holds for these parameters and put $r:=R/(10Kt)$. This choice of $r$ provides \eqref{n3} for the pair of balls $B(x_0, \rho t(1+\delta))$ and $B(x_0, \rho(1+c_1))$ and  \eqref{n3*} for $B(x_i, \rho t)$. Hence
%the inequality \eqref{eq:sl1} holds and \eqref{eq:sl2} gives $$t^{(1-c_2)(\tilde N (1+\varepsilon) +C)} \geq t^{N(1-\varepsilon)}.$$  We therefore have 
\begin{lemma}\label{corollary Cac}
   Let  $u$ be the solutions of (\ref{bi-ci}). Then there exists $C$ such that
\begin{equation}\label{caccioppoli no bdry}
    \|\nabla \triangle u \|_{L^{2}({B}^+_{1/2})} +    \|\nabla ^2 u\|_{L^{2}({B}^+_{1/2})} + \|\nabla  u\|_{L^{2}({B}^+_{1/2})}  \leq C \big(M\|u\|_{L^{2}({B}^{+}_{2})} + \sum^3_{j=0}\|\frac{\partial^j u}{\partial \nu^j}\|_{H^{3-j}(\bf{B}_2)}\big).
\end{equation}
\end{lemma}

\begin{proof}
We choose a smooth cut-off function $\eta(x)$ such that $\eta = 1$ in $B_{1}^{+} $, $\eta = 0$ outside $B_{2}^{+}$ and $|\nabla^k \eta| \leq {C}$ for $0\leq k\leq 4. $
By Reilly's formula, we have
\begin{align}
 \int_{B_{2}^{+}} |\nabla^2 (u\eta^2)|^2&= \int_{B_{2}^{+}} |\triangle (u\eta^2)|^2 \nonumber -Ric( \nabla (u\eta^2),  \nabla (u\eta^2))\\ &-
\int_{ \bf{B}_2} 2\frac{\partial{(u\eta^2 )}}{\partial \nu}\triangle (u\eta^2)-H \frac{\partial{(u\eta^2 )}}{\partial \nu}
 +h(  \nabla (u\eta^2),  \nabla (u\eta^2)),
  \label{grad-1}
\end{align}
 where $H$ is the mean curvature,  $h$ is the second fundamental form of the boundary, and Ric is the Ricci curvature tensor,
% +\int_{ \bf{B}_2} \nabla_i( u\eta^2)\nabla_{ij}(u \eta^2)\nu_j ds- \int_{ \bf{B}_2}\triangle ( u\eta^2)\nabla ( u\eta^2)\cdot \nu ds,
%The application of integration by parts show that
%\begin{align}
% \int_{B_{2}^{+}} |\nabla^2 (u\eta^2)|^2= \int_{B_{2}^{+}} |\triangle (u\eta^2)|^2 +\int_{ \bf{B}_2} \nabla_i( u\eta^2)\nabla_{ij}(u \eta^2)\nu_j ds- \int_{ \bf{B}_2}\triangle ( u\eta^2)\nabla ( u\eta^2)\cdot \nu ds,
% \label{grad-1}
%\end{align}
and $\nu=\langle \nu_1, \nu_2, \cdots, \nu_n\rangle $ is the unit outer normal on the boundary. Note that the mean curvature, Ricci curvature and second fundamental form are bounded since we are considering the compact smooth manifold  $\mathcal{M}$.
To estimate the terms in right hand side of (\ref{grad-1}), we need to study $\int_{B_{2}^{+}} |\triangle (u\eta^2)|^2$.
The following identity holds 
\begin{align}
\triangle (u \eta^2)  \triangle (u \eta^2)&=\triangle u \triangle (u\eta^4) +u\triangle (u\eta^2)\triangle \eta^2+4|\nabla u \cdot \nabla \eta^2|^2 \nonumber \\
&+2u (\nabla u\cdot \nabla \eta^2)\triangle \eta^2-u\triangle u\eta^2 ( 8|\nabla \eta|^2+\triangle \eta^2).
\label{shen-i}
\end{align}
See e.g. \cite{Shen06}. Now we study every term in the left hand side of (\ref{shen-i}) using the integration by parts arguments.
For the first term, performing integration by parts gives that
\begin{align}
\int_{B_{2}^{+}} \triangle  u \triangle (u\eta^4)-\int_{B_{2}^{+}} \triangle^2  u  (u\eta^4)=\int_{\bf{B}_2}\triangle u\frac{\partial (u\eta^4)}{\partial \nu} ds -\int_{\bf{B}_2}\frac{\partial \triangle u}{\partial \nu} (u\eta^4) ds.
\end{align}
Thus,  the equation (\ref{bi-ci}) yields that
\begin{align}
\int_{B_{2}^{+}} \triangle  u \triangle (u\eta^4)=\int_{B_{2}^{+}} W(x) u^2  \eta^4+
\int_{\bf{B}_2}\triangle u\frac{\partial (u\eta^4)}{\partial \nu} ds -\int_{\bf{B}_2}\frac{\partial \triangle u}{\partial \nu} (u\eta^4) ds.    
\end{align}
It follows from Cauchy-Schwartz inequality that
\begin{align}
   |\int_{B_{2}^{+}} \triangle  u \triangle (u\eta^4) |\leq CM \int_{B_{2}^{+}} u^2 +
   +C\sum^3_{j=0}\|\frac{\partial^j u}{\partial \nu^j}\|^2_{H^{3-j}(\bf{B}_2)}.
   \label{BB-2}
\end{align}
For the second term, by Young's inequality, for any small $\varepsilon>0$, we have
\begin{align}
 |\int_{B_{2}^{+}} u\triangle (u\eta^2) \triangle \eta^2|&\leq 
 \varepsilon\int_{B_{2}^{+}} |\triangle (u\eta^2) |^2+ \frac{C} {\varepsilon} \int_{B_{2}^{+}} (u \triangle \eta^2)^2 \nonumber \\
 &\leq 
 \varepsilon\int_{B_{2}^{+}} |\triangle (u\eta^2) |^2+ \frac{C} {\varepsilon} \int_{B_{2}^{+}} u^2.
 \label{com-3}
\end{align}
Thus, we can incorporate the first term in the right hand side (\ref{com-3}) into the left hand side when taking integration on the identity (\ref{shen-i}).
For the fourth term in the right hand side of (\ref{shen-i}), integrating by parts shows that
\begin{align}
2\int_{B_{2}^{+}} u(\nabla u\cdot \nabla \eta^2)\triangle \eta^2&=   \int_{B_{2}^{+}} (\nabla u^2\cdot \nabla \eta^2)\triangle \eta^2 \nonumber\\
&=-\int_{B_{2}^{+}} u^2(\triangle \eta^2)^2- \int_{B_{2}^{+}} u^2 \nabla \eta^2 \cdot \nabla (\triangle \eta^2)-\int_{\bf{B}_2} u^2 \triangle \eta^2 \nabla \eta^2\cdot \nu.
\end{align}
Then 
\begin{align}
2|\int_{B_{2}^{+}} u(\nabla u\cdot \nabla \eta^2)\triangle \eta^2|  \leq {C} \int_{B_{2}^{+}} u^2+ \|u\|^2_{L^2(\bf{B}_2)}.
\label{com-2}
\end{align}

For the last term in the right hand side of (\ref{shen-i}),
let $\phi=8|\nabla \eta|^2+\triangle \eta^2.$ Then
\begin{align}
\int_{B_{2}^{+}} \eta^2 u\triangle u\phi&=-\int_{B_{2}^{+}} |\nabla u|^2\eta^2 \phi-\frac{1}{2}\int_{B_{2}^{+}}\nabla (\eta^2\phi)\cdot\nabla u^2+\int_{\bf{B}_2}\eta^2u\phi\nabla u\cdot \nu ds\nonumber \\
&=-\int_{B_{2}^{+}} |\nabla u|^2\eta^2 \phi+ \frac{1}{2}\int_{B_{2}^{+}}\triangle (\eta^2\phi) u^2 -\int_{\bf{B}_2} \nabla (\eta^2\phi)\cdot \nu u^2 ds+\int_{\bf{B}_2}\eta^2u\phi\nabla u\cdot \nu ds.
\end{align}
Hence we obtain
\begin{align}
| \int_{B_{2}^{+}} \eta^2 u\triangle u\phi| &\leq |\int_{B_{2}^{+}} |\nabla u|^2\eta^2 \phi|+    \int_{B_{2}^{+}} u^2+ \sum^3_{j=0}\|\frac{\partial^j u}{\partial \nu^j}\|^2_{H^{3-j}(\bf{B}_2)}.
\label{prepare-1}
\end{align}

Now we study the third term in the right hand side of (\ref{shen-i}). It holds that $|\nabla u\cdot \nabla \eta^2|^2= 4 \eta^2 u_i u_j \eta_i \eta_j $, where $u_i$ is the covariant derivative of $u$ with respect to $\frac{\partial}{\partial x_i}$ and  the summation is understood by Einstein notation.
%4\eta^2 \frac{\partial u}{\partial x_i}\frac{\partial u}{\partial x_j} \frac{\partial \eta}{\partial x_i} \frac{\partial \eta}{\partial x_j}.$
Let $ \eta_i \eta_j =\psi$.
%\frac{\partial \eta}{\partial x_i} \frac{\partial \eta}{\partial x_j}=\psi$.
By the fact that
\begin{align}
u_i u_j \eta^2 \psi= ( (u\eta^2)_j  u\psi)_i-
(u\eta^2)_{ij} u \psi-u_j \psi_i \eta^2 u- u_i (\eta^2)_j u\psi- (\eta^2)_j \psi_i u^2,
\end{align}
%\frac{\partial u}{\partial x_i} \frac{\partial u}{\partial x_j}\eta^2 \psi= \frac{\partial}{x_i}\big( \frac{\partial (u\eta^2)}{\partial x_j} u\psi \big)- \frac{\partial^2(u\eta^2)}{ \partial x_i \partial x_j}
%    u\psi-\frac{\partial u}{\partial x_j} \frac{\partial \psi}{\partial x_i}\eta^2 u - \frac{\partial u}{\partial x_i} \frac{\partial \eta^2 }{\partial x_j} u\psi -\frac{\partial \eta^2 }{\partial x_j}\frac{\partial \psi }{\partial x_i}u^2 ,
we get
\begin{align}
\int_{B_{2}^{+}} u_i u_j\eta^2 \psi&= \int_{\bf B_{2}} (u\eta^2)_j u \psi \nu_i ds -\int_{B_{2}^{+}} (u\eta^2)_{ij} u\psi +\frac{1}{2} \int_{B_{2}^{+}} u^2 (\eta^2  \psi_i)_j \nonumber \\ &-\frac{1}{2} \int_{\bf B_{2}} u^2\eta^2  \psi_i \nu_j ds+\frac{1}{2}\int_{B_{2}^{+} } u^2 (  (\eta^2)_j\psi)_i- \frac{1}{2} \int_{\bf B_{2}} u^2\psi  (\eta^2)_j \nu_i ds
\nonumber \\ &- \int_{B_{2}^{+} }u^2  (\eta^2)_j \psi_i. 
\end{align}
%\int_{B_{2}^{+}} \frac{\partial u}{\partial x_i} \frac{\partial u}{\partial x_j}\eta^2 \psi&= \int_{\bf B_{2}} \frac{\partial (u\eta^2) }{\partial x_j} u \psi \nu_i ds -\int_{B_{2}^{+}} \frac{\partial^2(u\eta^2)}{ \partial x_i \partial x_j} u\psi +\frac{1}{2} \int_{B_{2}^{+}} u^2 \frac{\partial}{\partial x_j}(\eta^2 \frac{\partial \psi}{\partial x_i}) \nonumber \\ &-\frac{1}{2} \int_{\bf B_{2}} u^2\eta^2 \frac{\partial \psi}{\partial x_i}\nu_j ds+\frac{1}{2}\int_{B_{2}^{+} } u^2 \frac{\partial}{\partial x_i}( \frac{\partial \eta^2}{\partial x_j}\psi)- \frac{1}{2} \int_{\bf B_{2}} u^2\psi \frac{\partial \eta^2}{\partial x_j} \nu_i ds
%\nonumber \\ &- \int_{B_{2}^{+} }u^2 \frac{\partial \eta^2 }{\partial x_j}\frac{\partial \psi }{\partial x_i}.  
%\end{align}
By  Young's inequality, for any small $\varepsilon>0$, we have
\begin{align}
  \int_{B_{2}^{+}} (u\eta^2)_{ij} u\psi \leq \varepsilon   \int_{B_{2}^{+}} |\nabla^2 (u\eta^2)|^2 +\frac{C}{\varepsilon} \int_{B_{2}^{+}} |u|^2 |\nabla \eta|^4.
\end{align}
It follows from the last two inequalities that
\begin{align}
 \int_{B_{2}^{+}} u_i u_j \eta^2 \psi \leq    \varepsilon   \int_{B_{2}^{+}} |\nabla^2 (u\eta^2)|^2+  \frac{C}{\varepsilon} \int_{B_{2}^{+}} |u|^2 +  \sum^3_{j=0}\|\frac{\partial^j u}{\partial \nu^j}\|^2_{H^{3-j}(\bf{B}_2)}.
 \label{integrate-1}
\end{align}
That is,
\begin{align}
 \int_{B_{2}^{+}} |\nabla u\cdot \nabla \eta^2|^2 \leq    \varepsilon   \int_{B_{2}^{+}} |\nabla^2 (u\eta^2)|^2+  \frac{C}{\varepsilon} \int_{B_{2}^{+}} |u|^2 +  \sum^3_{j=0}\|\frac{\partial^j u}{\partial \nu^j}\|^2_{H^{3-j}(\bf{B}_2)}.
    \label{com-1} 
\end{align}
Recall that $\phi=8|\nabla \eta|^2+\triangle \eta^2.$ By the similar argument as (\ref{integrate-1}), we can also show that
\begin{align}
    \int_{B_{2}^{+}} |\nabla u|^2\eta^2 \phi \leq    \varepsilon   \int_{B_{2}^{+}} |\nabla^2 (u\eta^2)|^2+  \frac{C}{\varepsilon} \int_{B_{2}^{+}} |u|^2 +  \sum^3_{j=0}\|\frac{\partial^j u}{\partial \nu^j}\|^2_{H^{3-j}(\bf{B}_2)}.
    \label{com-1-2}
\end{align}

Therefore, from (\ref{prepare-1}), we derive that
\begin{align}
| \int_{B_{2}^{+}} \eta^2 u\triangle u\phi|    \leq    \varepsilon   \int_{B_{2}^{+}} |\nabla^2 (u\eta^2)|^2+  \frac{C}{\varepsilon} \int_{B_{2}^{+}} |u|^2 +  \sum^3_{j=0}\|\frac{\partial^j u}{\partial \nu^j}\|^2_{H^{3-j}(\bf{B}_2)}.
\label{com-0}
\end{align}
 The following interpolation inequality holds, see e.g. \cite{Au98},
\begin{align}
    \int_{B_2^+} |\nabla (u\eta^2)|^2\leq 
{\varepsilon}\int_{B_2^+} |\nabla^2 (u\eta^2) |^2+ 
  \frac{C}{\varepsilon}  \int_{B_2^+} |u\eta^2|^2
\label{inte-1}
\end{align}
for any small $\varepsilon>0$.
By choosing $\varepsilon>0$ small enough, and combining terms in (\ref{grad-1}), (\ref{shen-i}), (\ref{BB-2}), (\ref{com-3}),  (\ref{com-2}), (\ref{com-1}), (\ref{com-0}), and (\ref{inte-1}),  we obtain that
\begin{align}
  \int_{B_2^+} |\nabla^2 (u\eta^2) |^2\leq   {C}M \int_{B_{2}^{+}} |u|^2 +  \sum^3_{j=0}\|\frac{\partial^j u}{\partial \nu^j}\|^2_{H^{3-j}(\bf{B}_2)}.
\end{align}

%We multiply (\ref{bi-ci}) by $\eta^4 u$, where $\eta = 1$ in $B_{1}^{+} $, $\eta = 0$ outside $B_{2}^{+}$ and $|\nabla \eta| \leq {C}$. 
 % Since $u$ is a solution of \eqref{bi-ci}, we can multiply \eqref{bi-ci} by $\eta^4 u$ and integrate over $B_2^{+}$ to get
%\begin{align}
%\int_{B_2^+}\triangle u\triangle (\eta^4 u)+\int_{\bf{B}_2}\frac{\partial\triangle u}{\partial \nu}\eta^4 u-\int_{\bf{B}_2}{\triangle u} \frac{\partial (\eta^4 u)}{\partial \nu}=\int_{B_2^+} W(x)\eta^4 u^2.
%\end{align}

%Integrating by parts show that
%\begin{align}

%By Young's inequality, we get
%\begin{align}
%\int_{B_2^+}  \eta^4 (\triangle u)^2\leq \varepsilon \int_{B_2^+}  \eta^4 (\triangle u)^2+ \frac{C}{\varepsilon}
%\int_{B_2^+} |\nabla u|^2 +\frac{C}{\varepsilon}
%\int_{B_2^+} u^2+\sum^3_{j=0}\|\partial^j_{x_n}u\|^2_{H^{3-j}(\bf{B}_2)}.
%\end{align}
%By choosing $\varepsilon$ small,
%\begin{align}
%\int_{B_2^+}  \eta^4 (\triangle u)^2\leq C 
%\int_{B_2^+} |\nabla u|^2 +{C}M^2
%\int_{B_2^+} u^2+\sum^3_{j=0}\|\partial^j_{x_n}u\|^2_{H^{3-j}(\bf{B}_2)}.
%\end{align}
%Since $\eta=0$ on $\partial B_2^+\backslash \{x_n=0\}$, the following Poincare inequality holds
%\begin{align}
%    \int_{B_2^+} |\nabla (u\eta^2)|^2\leq C  \int_{B_2^+} |\nabla |\nabla (u\eta^2)||^2\leq C 
%\int_{B_2^+} |\nabla^2 (u\eta^2) |^2.
%\end{align}
Thanks to (\ref{inte-1}),  we further show that
\begin{align}
 \int_{B_2^+} |\nabla (u\eta^2) |^2+  \int_{B_2^+} |\nabla^2 (u\eta^2) |^2\leq   {C}M \int_{B_{2}^{+}} |u|^2 +  \sum^3_{j=0}\|\frac{\partial^j u}{\partial \nu^j}\|^2_{H^{3-j}(\bf{B}_2)}.
 \label{com-a}
\end{align}
By the definition of $\eta$, it is easy to see that
\begin{align}
\int_{B_1^+} |u |^2 +  \int_{B_1^+} |\nabla^2 u |^2 \leq {C}M \int_{B_{2}^{+}} |u|^2 +  \sum^3_{j=0}\|\frac{\partial^j u}{\partial \nu^j}\|^2_{H^{3-j}(\bf{B}_2)}.  
\label{zhu-zhu}
\end{align}

At last, we want to provide a upper bound for $\int_{B_{1/2}^+} |\nabla \triangle u |^2$. We choose a smooth cut-off function $\hat{\eta}$ such that
$\hat{\eta}=1$ in ${B_{1/2}^+}$, $\hat{\eta}=0$ outside of ${B_{1}^+}$, and $|\nabla^k \hat{\eta}| \leq {C}$ for $0\leq k\leq 4. $
Multiplying (\ref{bi-ci}) by $\triangle u \hat{\eta}^4$, and then integrating by parts give that
\begin{align}
 \int_{B_1^+} |\nabla \triangle u \hat{\eta}^2 |^2
+ 4\int_{B_1^+} \nabla \triangle u \cdot\nabla \hat{\eta} \hat{\eta}^3 \triangle u-\int_{\bf B_1} \nabla \triangle u \cdot \nu \hat{\eta}^4 \triangle u =- \int_{B_1^+} W(x) u\triangle u \hat{\eta}^4.
\label{tri-2}
\end{align}
By Young's inequality, for any small $\varepsilon>0$, we have
\begin{align}
  4\int_{B_1^+} \nabla \triangle u \cdot\nabla \hat{\eta}\hat{\eta}^3 \triangle u\leq   \varepsilon\int_{B_1^+} |\nabla \triangle u \hat{\eta}^2 |^2 +\frac{C}{\varepsilon} \int_{B_1^+} |\nabla \hat{\eta}|^2 |\hat{\eta}|^2 |\triangle u|^2.
  \label{tri-3}
\end{align}
It follows from (\ref{zhu-zhu})--(\ref{tri-3}), and Cauchy-Schwartz inequality that
\begin{align}
\int_{B_1^+} |\nabla \triangle u \hat{\eta}^2 |^2&\leq  {C}M^2 \int_{B_{1}^{+}} |u|^2 +  \sum^3_{j=0}\|\frac{\partial^j u}{\partial \nu^j}\|^2_{H^{3-j}(\bf{B}_1)}+ C \int_{B_1^+} |\triangle u|^2 \nonumber \\   &\leq {{C}M^2} \int_{B_{2}^{+}} |u|^2 +  \sum^3_{j=0}\|\frac{\partial^j u}{\partial \nu^j}\|^2_{H^{3-j}(\bf{B}_2)},
\end{align}
where we have chosen $\varepsilon$ small enough in (\ref{tri-3}).
Thus, it holds that
\begin{align}
\int_{B_{1/2}^+} |\nabla \triangle u|^2&\leq  {C}M^2 \int_{B_{2}^{+}} |u|^2 +  \sum^3_{j=0}\|\frac{\partial^j u}{\partial \nu^j}\|^2_{H^{3-j}(\bf{B}_2)}.
\end{align}
Hence we arrive at the desired results in the lemma from the last inequality and (\ref{zhu-zhu}).
\end{proof}

At last, we prove the simplex lemma  using the almost monotonicity of doubling index in Lemma \ref{baa}.
The following Euclidean geometry result holds:
 {\it there exist $\hat{\tau}>0$, $K \geq 2/\gamma$ depending on $\gamma, n$ only  such that if $\rho = K diam(S)$, then
 $B_{\rho(1+\hat{\tau})}(x_0) \subset \cup_{i=1}^{n+1}B_{\rho}(x_i)$.} 

 We remark that if the simplex is very degenerate ($\gamma$ is small), then $\hat{\tau}$ has to be  small and the number $K$ has to be big: $$ \hat{\tau} \to 0, K \to +\infty  \textup{ as } \gamma \to 0.$$

\begin{proof}[Proof of Lemma \ref{simplex}]
 
Due to the almost monotonicity of the doubling index,  we may assume that $B_i$ have the same radius $\rho =K diam(S)$.
Let $\hat{M}= \|u\|_{L^\infty(\cup^{n+1}_{i=1} B_\rho(x_i))}$. Thus, there exists some $i$ such that $\hat{M}= \|u\|_{L^\infty( B_\rho(x_i))}$. By the Euclidean geometric result, we have $\|u\|_{L^\infty(B_{\rho(1+\hat{\tau})}(x_0))}\leq \hat{M}$. The monotonicity of doubling index in (\ref{mono-del-1}) implies that $ \|u\|_{L^\infty( B_{t\rho}(x_i))}\geq \hat{M} t^{N(1-\delta)}$ for $N\geq N_0(g, \delta)$ and $t>2$.

The following geometric fact also holds.
{\it there exists ${c_1}/{t} \in (0,1)$ such that $B_{\rho t}(x_i) \subset B_{\rho t(1+c_1/t)}(x_0) $ as $t \to +\infty$. \it}
This result can be shown by the triangle inequality.  For $
x\in B_{\rho t} (x_i)$, we can show that
\begin{align*}
|x-x_0|\leq |x-x_i|+|x_i-x_0|\leq {\rho t}+diam(S)={\rho t}(1+c_1/t),
\end{align*}where $c_1=1/K$.

Let $\hat{N}$ be the doubling index for $B_{\rho t(1+c_1/t))}(x_0)$. By the almost monotonicity of doubling index (\ref{mono-del}), we have
$$
\left[\frac{t(1+c_1/t)}{1+\hat{\tau}}\right]^{\hat{ N} (1+\delta) } \geq \frac{\|u\|_{L^\infty(B_{\rho t(1+c_1/t)}(x_0)} }{\|u\|_{L^\infty(B_{\rho(1+\hat{\tau})}(x_0)}} \geq  \frac{\|u\|_{L^\infty(B_{\rho t}(x_i))}}{\|u\|_{L^\infty(B_{\rho }(x_i))}} \geq  t^{N(1-\delta)}.
$$
Thus, it implies that
\begin{equation} \label{eq:sl1}
\left[\frac{t(1+c_1/t)}{1+\hat{\tau}}\right]^{\hat{ N} (1+\delta) } \geq t^{N(1-\delta)}.
\end{equation}

  Now, we specify our choice of parameters. For $t>2$ large, there exists $c_2\le \ln\frac{1+\hat{\tau}}{1+\hat{\tau}/2}/\ln t$ such that
\begin{equation} \label{eq:sl2}
\frac{t(1+c_1/t)}{1+\hat{\tau}} \leq t^{1-c_2}.
\end{equation}
 Next we choose $\delta=o(c_2)$  (i.e. say $\delta=t^{-\beta}$ for some large $\beta$) such that
\begin{equation} 
 \frac{1-\delta}{(1+\delta)(1-c_2)}> 1+c_3
\end{equation}
for some $c_3(c_2)>0$.
Thanks to (\ref{eq:sl1}) and (\ref{eq:sl2}), we can estimate $\hat{N}$ from below
\begin{align}
\hat{ N} \geq N \frac{(1-\delta)}{(1+\delta)(1-c_2)} \geq N(1+c_3).
%- \frac{C\ln t/\delta}{(1+\delta)(1-c_2)} \geq N(1+2c_3) - \frac{C\ln t/\delta}{(1+\delta)(1-c_2)} \nonumber \\ 
%+ (c_3 N_0 - \frac{C\ln t/\delta}{(1+\delta)(1-c_2)}).
\end{align}
%We can also choose $N_0$ to be big enough so that $c_3 N_0 - \frac{C\ln t/\delta}{(1+\delta)(1-c_2)}>0$. Thus, $\hat{ N} > N(1+c_3)$.
By fixing $t$ large, then $c_2$, $\delta$ and $c_3$ depend on $\hat{\tau}$. On the other hand, $\hat{\tau}$ depends on $\gamma$ and $n$.
To apply the almost monotonicity of doubling index in Lemma \ref{baa}, we choose $r\leq \frac{\bar R(\delta, g)}{100Kt}$, where $\bar R(\delta, g)$ is in Lemma \ref{baa}. Thus, $C$ and $c$ in (\ref{simplex-mon}) depend on $\gamma$ and $n$, while $N_0$ and $r$ depend on $\gamma$ and $g$.
\end{proof}
\bibliography{mybib}
\bibliographystyle{abbrv}
\end{document}